\documentclass[12pt,reqno]{amsart}
\usepackage{graphicx} 
\usepackage{float}
\usepackage{amsmath}
\usepackage{amssymb}
\usepackage{mathrsfs}
\usepackage{mathtools}
\usepackage{booktabs}
\usepackage{appendix}
\usepackage[a4paper,top=2.5cm,bottom=2.5cm,left=2cm,right=2cm,bindingoffset=0.5cm]{geometry}
\usepackage{hyperref}
\numberwithin{equation}{section}
\allowdisplaybreaks

\newcommand{\ud}{\,\mathrm{d}}
\newcommand{\udiv}{\, \mathrm{div}}

\newcommand{\R}{{\mathbb{R}}}

\title [Kinetic equation from the large-scale limit of Cucker--Smale model]{On the kinetic equation arising from the large-scale limit of the Cucker-Smale model}

\thanks{The work of Ruicheng Cheng and Zhenfu Wang was partially supported by the National Key R\&D Program of China (Project No.~2024YFA1015500) and by the NSFC (Grant Nos.~12595282 and 12171009). The work of Seung-Yeal Ha was supported by National Research Foundation (NRF) grant funded by the Korean government (MSIT) (RS-2025-00514472). }

\author[Cheng]{Ruicheng Cheng}
\address[Ruicheng Cheng]{\newline School of Mathematical Sciences, \newline
Peking University, Beijing, 100871, the People's Republic of China}
\email{chengruicheng02@stu.pku.edu.cn}

\author[Ha]{Seung-Yeal Ha}
\address[Seung-Yeal Ha]{\newline Department of Mathematical Sciences and Research Institute of Mathematics, \newline
Seoul National University, Seoul, 08826, Republic of Korea}
\email{syha@snu.ac.kr}

\author[Lee]{Jaemoon Lee}
\address[Jaemoon Lee]{\newline Department of Mathematical Sciences, \newline
Seoul National University, Seoul, 08826, Republic of Korea}
\email{dlwoans0001@snu.ac.kr}

\author[Wang]{Zhenfu Wang}
\address[Zhenfu Wang]{\newline Beijing International Center for Mathematical Research, \newline Peking University, Beijing, 100871, the People's Republic of China}
\email{zwang@bicmr.pku.edu.cn}

\begin{document}
\newtheorem{theorem}{Theorem}[section]
\newtheorem{lemma}[theorem]{Lemma}
\newtheorem{corollary}[theorem]{Corollary}
\newtheorem{proposition}[theorem]{Proposition}
\newtheorem{remark}[theorem]{Remark}
\newtheorem{definition}[theorem]{Definition}
\newtheorem{example}[theorem]{Example}

\begin{abstract}
     We propose a large-scale scaling viewpoint for deriving mesoscopic dynamics from interacting particle systems and  apply it to the Cucker--Smale flocking model. In contrast with the classical mean-field regime leading to the
   Vlasov-type Cucker--Smale equation with spatially nonlocal (convolution) alignment force, our scaling yields
   a kinetic equation whose alignment field becomes local in space and nonlocal only in velocity. 	For the spatially homogeneous case, we obtain an explicit solution and derive quantitative flocking rates.
   For the spatially inhomogeneous equation we establish a local well-posedness in $W^{1,\infty}$ and in
   $C_b^{1,\alpha}$, highlighting the additional difficulties caused by the absence of a convolution structure.
   Moreover, for sufficiently small interaction strength we present a global well-posedness and a forward-in-time
   $L^1$ asymptotic completeness property. Finally, we investigate mono-kinetic solutions and exhibit finite-time
   blow-up scenarios.
\end{abstract}

\maketitle


\tableofcontents

\section{Introduction}
\subsection{Motivation: A first-order system}
It is well established that deriving mesoscopic dynamics from microscopic dynamics involves specific rescaling of time, space, and interaction strength, each yielding distinct kinetic descriptions of the underlying microscopic dynamics. For classical Newtonian systems, for instance, standard mean-field scaling leads to the Vlasov equation~\cite{Dob79}. In contrast, the low-density scaling associated with hard-sphere collisions gives rise to the Boltzmann equation~\cite{PS17}, while weak-coupling scaling recovers the Landau equation~\cite{BPS13}. Broadly, these limits fall under the category of macroscopic scalings, wherein the system is observed over large space-time scales to investigate the collective statistical behaviors of interacting particles. \par 

To motivate the large scale limit considered in this paper,  we  start from  a simple 1-dimensional first-order $N$-particle system:
\[ \frac{\mathrm{d}q_i}{\mathrm{d}\tau}=\sum_{j = 1}^N \phi(|q_i-q_j|),\quad i \in [N] := \{ 1,\cdots,N \}, \]
where $q_i \in \mathbb{R}$  stands for the position of the $i$-th particle and $\phi$ is a compactly supported interaction function. If we rescale the above system by
$$x_i=\varepsilon q_i,\quad t=\varepsilon\tau$$
with the order $N\varepsilon=1$, then we get the rescaled system
\[
\frac{\mathrm{d}x_i}{\mathrm{d}t} = \sum_{j = 1}^N  \phi\left(\frac{|x_i-x_j|}{\varepsilon}\right),\quad i \in [N].
\]
Suppose that the interaction term is the indicator function $\phi(x)=\mathbf{1}_{|x|< 1}$.
Then the velocity of a particle equals the number of  particles within distance $\varepsilon$ of it:
\[
\dot{x}_i = \#\{j\in [N] \mid |x_i-x_j|<\varepsilon \},
\]
where $ \# A$ is the cardinality of the set $A$. If we introduce an empirical measure
\[
\mu_t^N = \frac{1}{N}\sum_{i=1}^N \delta_{x_i(t)},
\]
then the equation becomes
\[
\dot{x}_i = N\int_{x_i-\varepsilon}^{x_i+\varepsilon} \mathrm{d}\mu_t^N.
\]
Assume that the empirical measure $\mathrm{d}\mu_t^N$ converges to a density $f(t,x) \mathrm{d}x$ as $N\rightarrow\infty$. Then using the fact that $N=\varepsilon^{-1}$,
formally the above equation will converge to 
$$\dot{x} = 2f(t,x),$$
which we interpret it as the characteristic curve of the continuity equation solved  by $f$, i.e. 
$$\partial_t f  + \partial_x(2f^2)=0,$$
which is exactly the inviscid Burgers' equation in $\mathbb{R}$. Comparison with the usual mean-field limit, the above scaling leading to the (inviscid) Burgers equation is quite different from the standard mean-field regime
reviewed for instance in \cite{JabinWangActiveParticles2017,SznitmanSaintFlour1991}.
In the classical mean-field limit, one considers weak, long-range interactions of size $1/N$, for instance
\[
\dot x_i(t)=\frac1N\sum_{j=1}^N K\bigl(x_i(t)-x_j(t)\bigr)
\qquad\text{or}\qquad
\ud X_t^i=\frac1N\sum_{j=1}^N K\bigl(X_t^i-X_t^j\bigr)\, \ud t+\sqrt{2\nu} \, \ud B_t^i,
\]
where $\{B_t^i \}$ is the collection of i.i.d. Brownian motions. Note that the drift felt by a tagged particle is an averaged field generated by the empirical measure
$\mu_t^N=\frac1N\sum_{j=1}^N\delta_{x_j(t)}$. Formally (and rigorously under suitable assumptions) one obtains
a nonlocal McKean--Vlasov / Vlasov--Fokker--Planck type equation:
\[
\partial_t \rho+\udiv_x \!\bigl(\rho\,(K*\rho)\bigr)=\nu\,\Delta_x \rho,
\]
whose characteristic velocity involves the convolution $K*\rho$ rather than the pointwise density.
In contrast, our scaling is \emph{local} and \emph{moderately interacting}: the range $\varepsilon\to0$ while
$N\varepsilon = O (1)$, and the (un-normalized) counting interaction
$\phi=\mathbf{1}_{|x|<1}$ yields the convergence towards the inviscid Burgers equation, where shocks and entropy selection are intrinsic.
This short-range / moderate-interaction paradigm is closer in spirit to the program initiated by
Oelschl\"ager \cite{Oelschlager1985}, where the interaction range shrinks with $N$ and the limiting PDE may change
character compared with the classical mean-field closure.
Finally, Sznitman also discussed stochastic particle approximations of Burgers-type dynamics:
adding Brownian noise produces a viscous regularization and leads formally to the viscous Burgers equation:
\[
\partial_t f + \,\partial_x(2f^2)=\nu\,\partial_{xx}f,
\]
and a propagation-of-chaos derivation in this direction is developed in
\cite{SznitmanBurgers1986} (see also the Saint-Flour notes \cite{SznitmanSaintFlour1991}).
\subsection{A large scale limit}
In nature, birds~\cite{BCC+08}, fish~\cite{TKI+13}, fireflies~\cite{Buc68} are observed to exhibit flocking or synchronization phenomena/behavior. 
These observations have motivated a variety of mathematical models, including the Kuramoto model~\cite{Kur05,Kur84} for synchronization and self-propelled particle models such as the Vicsek model~\cite{VCB+95,VZ12} or the Cucker--Smale (CS) model~\cite{CS07} for collective flocking. In this paper,  we focus on the CS model.

The CS model is well known in the literature of flocking, which can be used to describe a wide class of collective behaviors. 
It can be written as  the following system of ordinary differential equations:
\begin{equation}\label{CS}
    \begin{cases}
   \displaystyle  \dfrac{\mathrm{d}q_i}{\mathrm{d}\tau}=v_i, \quad \tau > 0, \quad i \in [N], \\
    \displaystyle\frac{\mathrm{d}v_i}{\mathrm{d}\tau}=k\sum_{j=1}^{N}\psi\left(|q_j-q_i|\right)(v_j-v_i),
\end{cases}
\end{equation}
where $q_i,v_i\in\mathbb{R}^d$ denote the position and velocity of the $i$-th particle respectively, $k$ is a nonnegative coupling strength and $\psi$ is a communication weight function that is nonnegative, bounded, Lipschitz continuous and non-increasing.\par
There have been many extensions of the CS model, including the random failures of connection~\cite{HM19,RLX15}, modified communication protocols (e.g.\ normalization)~\cite{MT11}, adding stochastic perturbations~\cite{AH10,HLL09}, and more.
In particular, finding a limiting description of CS-type models in large-population regimes is an active research subject~\cite{HKZ18,HL09}.
In particular, the derivation and well-posedness of the mean-field limit of the CS model are well documented in literature~\cite{HL09,HT08}.\par

In the classical mean-field regime for the CS model, namely with coupling strength $k=\lambda/N$ for some $\lambda>0$,  one formally obtains the Vlasov-type equation:
\begin{equation}\label{eq:kineticCS}
\partial_t f +v\cdot\nabla_x f=\lambda\nabla_v\cdot\left(\tilde{E}[f]f\right),
\end{equation}
where the alignment force field $\tilde{E}[f]$ is given by 
$$\tilde{E}[f](t,x,v)=\int_{\mathbb{R}^{2d}}\psi(|x-x_*|)(v-v_*)f(t,x_*,v_*)\mathrm{d}x_*\mathrm{d}v_*.$$
Under suitable assumptions on the communication weight $\psi$, it is by now standard to rigorously justify the mean-field limit of the particle system~\eqref{CS} towards the kinetic equation~\eqref{eq:kineticCS};
see, for instance, Ha-Liu \cite{HL09} and Ha-Kim-Zhang \cite{HKZ18} for finite-time and uniform-in-time mean field using particle-in-cell method, Carrillo--Choi--Hauray--Salem \cite{CarrilloChoiHauraySalem2019} for a general framework
covering Cucker--Smale interactions with discontinuous (sharp) sensitivity regions, and Mucha--Peszek \cite{MuchaPeszek2018} for the kinetic Cucker--Smale equation with singular communication weights, etc.

In this paper, by rescaling the CS model \eqref{CS}, we derive a different kinetic equation with a spatially-local alignment interaction:
\begin{equation}\label{Eq}
\partial_t f +v\cdot\nabla_x f=\gamma\nabla_v\cdot\left(E[f]f\right),
\end{equation}
where $\gamma>0$ is a constant and
$$E[f](t,x,v)=\int_{\mathbb{R}^d}(v-v_*)f(t,x,v_*)\mathrm{d}v_*.$$ 
We would  refer to \eqref{Eq} as the \emph{kinetic CS equation in the moderate-interaction (shrinking-radius) regime},
to distinguish it from the classical mean-field (Vlasov--alignment) equation \eqref{eq:kineticCS}. 

First, we derive the equation~\eqref{Eq} heuristically using the method of scaling. If we rescale equation \eqref{CS} by
$$x_i=\varepsilon q_i,\qquad t=\varepsilon\tau,\qquad k=\varepsilon\kappa$$
with system size $N\varepsilon^d=1$, we get the rescaled CS model:
\begin{equation}\label{RCS}
	\begin{cases}
		\displaystyle \dfrac{\mathrm{d}x_i}{\mathrm{d}t}=v_i, \quad t > 0,~~i \in [N], \\
		\displaystyle\frac{\mathrm{d}v_i}{\mathrm{d}t}=\kappa\sum_{j=1}^{N}\psi\Big(\frac{|x_j-x_i|}{\varepsilon}\Big)(v_j-v_i),
	\end{cases}
\end{equation}
Assume that the coupling function $\psi$ is  compactly supported in $B(0, R)$ (the ball with a radius $R$, a center 0).  We now estimate the acceleration of the $i-$th particle: 
\[
\kappa\!\sum_{|x_j-x_i|<\varepsilon R}  \psi\Big(\frac{|x_j-x_i|}{\varepsilon}\Big) (v_j-v_i)=\kappa N\!\int_{B(x_i,\varepsilon R )\times \mathbb{R}^d} \psi\Big(\frac{|y-x_i|}{\varepsilon}\Big)  (w-v_i)\ \mu_t^N(\mathrm{d} y ,\mathrm{d} w),\]
where $\displaystyle \mu_t^N = \frac{1}{N}\sum_{i=1}^N \delta_{(x_i(t),v_i(t))}$. Assume that 
\[ \mu_t^N(\mathrm{d} y,\mathrm{d} w) \quad \mbox{converges to} \quad f(t,y,w) \mathrm{d} y \mathrm{d} w \quad \mbox{as $N\rightarrow\infty$}. \]
 When the limit density  $f$ is Lipschitz in $y$,  one can replace $f(t, y, w)$ by its value at the point $y=x_i$, i.e. $f(t,x_i,w)$,  for $y \in B(x_i, \varepsilon R)$. Then the right-hand side of the above equation can be approximated by 
\[
\kappa N\!\int_{B(x_i,\varepsilon R )\times \mathbb{R}^d} \psi\Big(\frac{|y-x_i|}{\varepsilon}\Big)  (w-v_i)\ f(t, x_i, w) \ud y \ud w.   
\]
By Fubini's theorem, we integrate first with respect to $y$ to obtain that 
\[
\frac{\ud v_i }{\ud t } \approx \kappa N\!\int_{ \mathbb{R}^d}  (w-v_i)\ f(t, x_i, w)  \Big\{    \int_{B(x_i,\varepsilon R )}  \psi\Big(\frac{|y-x_i|}{\varepsilon}\Big) \ud y  \Big\} \ud w. 
\]
By direct computations of the integration on the ball $B(x_i, \varepsilon R )$,  and using the scaling  $N \varepsilon^d = 1$, we obtain that 
\[
\kappa N   \int_{B(x_i,\varepsilon R )}  \psi\Big(\frac{|y-x_i|}{\varepsilon}\Big) \ud y = \kappa N \varepsilon^d \int_{\mathbb{R}^d} \psi(|z| ) \ud z  = \kappa  \int_{\mathbb{R}^d} \psi(|z| ) \ud z =: \gamma. 
\]
In summary,  the evolution of the tagged  particle  $i$ formally converges to 
\[
\frac{\ud  x }{\ud t }  = v, \quad \quad  \frac{\mathrm{d}v}{\mathrm{d}t}= \gamma \int_{\mathbb{R}^d} (w -v)f(t, x, w)\mathrm{d} w,
\]
which is exactly the characteristic curve of  our limit kinetic equation~\eqref{Eq}. We can also derive~\eqref{Eq} in a more standard way by using the BBGKY hierarchy together with a molecular chaos assumption. We postpone this discussion in Appendix \ref{App-A}.

In this paper, we are interested in the analysis of the kinetic equation \eqref{Eq}, including an explicit solution for the spatially homogeneous equation, local and global well-posedness results for the spatially inhomogeneous equation as well as the property of asymptotic completeness. In the last section, we also study the mono-kinetic solutions, where we show such solutions can blow up in finite time. We leave a rigorous justification of this scaling limit in our paper to future work. The equation \eqref{Eq} is similar to the one in~\cite{BCP97}, where the integrand in the non-linear part differs by a factor of $|v-v_*|$.
The difference in derivation is that they start from some collision model, whereas we start from the rescaled CS model. Our proofs of well-posedness results in this paper apply to that equation with small modification. Moreover, we extend their one-dimensional results to  any $d$-dimensional case.

\begin{remark}\textup{
	By choosing different scalings, one can obtain kinetic equations with different (and sometimes richer) interaction
	structures. For instance, in the \emph{weak-coupling} regime for a Newtonian $N$-particle system with a smooth
	short-range potential, one expects the (inhomogeneous) kinetic Landau equation to emerge.
	Following Bobylev--Pulvirenti--Saffirio~\cite{BPS13}, we consider
	\[
	\frac{d q_i}{d\tau}=v_i,\qquad
	\frac{d v_i}{d\tau}=\sum_{j\neq i} F_\varepsilon(q_i-q_j),\qquad i \in [N],
	\]
	introduce $\varepsilon\ll1$, rescale $x_i=\varepsilon q_i$ and $t=\varepsilon\tau$, assume unit macroscopic density
	$N \varepsilon^{d} = 1$, and weaken the interaction by $\phi_\varepsilon=\sqrt{\varepsilon}\,\phi$ (so
	$F_\varepsilon=\sqrt{\varepsilon}\,F$ with $F=-\nabla\phi$). In the macroscopic variables, this becomes
	\[
	\dot v_i(t)=\frac{1}{\sqrt{\varepsilon}}\sum_{j\neq i}F\!\left(\frac{x_i(t)-x_j(t)}{\varepsilon}\right),
	\]
	so that many small deflections accumulate into an effective diffusion in velocity, yielding formally a Landau-type
	collision operator; see \cite{BPS13} for a consistency result.}
\end{remark}	 
	
\vspace{0.2cm}

The rest of this paper is organized as follows. 
In Section \ref{sec:2}, we establish a priori estimates and gives an explicit formula in the spatially homogeneous case.
In Section \ref{sec:3}, we study local well-posedness in $W^{1,\infty}$ and $C_b^{1,\alpha}$, global well-posedness for small coupling, forward asymptotic completeness, and mono-kinetic solutions. 
In Section \ref{sec:4}, we provide several numerical evidence on conservation and dissipation trend. 
Also, we check the effect of different coupling strength.
Finally, Section \ref{sec:5} is devoted to a brief summary of the paper and discussion on some issues for a future work. 
In Appendix \ref{App-A}, we give a formal BBGKY derivation of the kinetic model.

\subsection{Gallery of Notations}
Throughout the paper, $|\cdot|$ denotes the Euclidean norm in the Euclidean space $\mathbb{R}^n$ of any dimension $n$. For a matrix $A\in\mathrm{M}_{n\times n}$, $|A|$ denotes the operator norm
$$|A|\triangleq\sup_{\substack{x\in\mathbb{R}^n\\|x|=1}}|Ax|.$$
We also write $z=(x,v)\in\mathbb{R}^d\times\mathbb{R}^d$, where $x$ and $v$ stand for the position and velocity variables respectively.\\[4 pt]
\textbf{$v$-support}\quad
For a function $f= f(x,v)$ on $\mathbb{R}^{2d}$, the \textit{$v$-support} of $f$ is defined as
$$v\textup{-}\mathrm{supp}(f)\triangleq\left\{v\in\mathbb{R}^d| f(x,v)\neq 0 \textup{ for some } x\in\mathbb{R}^d\right\}.$$
\textbf{The space $W^{1,\infty}$}\quad
Let  $W^{1,\infty}(\mathbb{R}^{2d})$ denote the Sobolev space consisting of all essentially bounded functions on $\mathbb{R}^{2d}$ with essentially bounded weak derivatives. It is well-known that every $W^{1,\infty}$ function has a Lipschitz continuous version, and for $f\in W^{1,\infty}$, $\Vert\nabla f\Vert_{L^{\infty}}$ is comparable with the Lipschitz constant of $f$. In this paper, when we write a function $f\in W^{1,\infty}(\mathbb{R}^{2d})$, we always refer to the Lipschitz continuous version of $f$, and for $f\in W^{1,\infty}(\mathbb{R}^{2d})$, we define
\[ \mathrm{Lip}(f)\triangleq\sup_{\substack{z_1,z_2\in\mathbb{R}^{2d}\\z_1\neq z_2}}\frac{|f(z_1)-f(z_2)|}{|z_1-z_2|}, \quad  \Vert f\Vert_{W^{1,\infty}}\triangleq\Vert f\Vert_{L^{\infty}}+\mathrm{Lip}(f). \]
\textbf{The spaces $C^{1,\alpha}$ and $C_b^{1,\alpha}$}\quad
For $\alpha\in(0,1]$, let $C^{1,\alpha}(\mathbb{R}^{2d})$ denote the H\"{o}lder space consisting of all $C^1$ functions on $\mathbb{R}^{2d}$ whose first derivatives are $\alpha$-H\"{o}lder continuous, and $C_b^{1,\alpha}(\mathbb{R}^{2d})$ denotes the collection of those functions in $C^{1,\alpha}(\mathbb{R}^{2d})$ that both itself and its first derivatives are bounded. For $f\in C_b^{1,\alpha}(\mathbb{R}^{2d})$, we denote
\[ [\nabla_zf]_{\alpha}\triangleq\sup_{\substack{z_1,z_2\in\mathbb{R}^{2d}\\z_1\neq z_2}}\frac{|\nabla_zf(z_1)-\nabla_zf(z_2)|}{|z_1-z_2|^{\alpha}}, \quad \Vert f\Vert_{C^{1,\alpha}}\triangleq\sup_{\mathbb{R}^{2d}}|f(z)|+\sup_{\mathbb{R}^{2d}}|\nabla_zf(z)|+[\nabla_zf]_{\alpha}. \]
\textbf{The generic constant $C$}\quad
Throughout the paper, the letter $C$ denotes a generic positive constant whose value may vary from line to line. Its dependence on initial data or parameters will be specified in the statements when needed.

\section{Preliminaries} \label{sec:2}
We begin with key physical observables of the kinetic equation.
Specifically, we track mass, momentum, energy, and entropy and $L^p$ norm, which  provide a macroscopic perspective on the system’s overall behavior.
Then these results are applied to the spatially homogeneous version of \eqref{Eq} to get an explicit solution formula.
\subsection{A priori estimates} \label{sec:2.1}
In this subsection, we provide a priori conservation and monotonicity identities for solutions to~\eqref{Eq}, obtained without using any explicit representation of the solution. First, we set $H$-functional:
\[ H[f]:= \int_{\R^{2d}} f\log f\,\ud x \ud v. \]
\begin{proposition}\label{evolution}
	Let $f= f(t,x,v)\ge 0$ be a smooth solution to~\eqref{Eq} with a sufficient fast decay as $|x|+|v|\to\infty$
	so that all integrations by parts below are justified. Then, we have the following estimates:
	\begin{enumerate}
		\item[\textup{(1)}] \textup{(Conservation laws)}:
		\[
		\frac{\ud}{\ud t}\int_{\R^{2d}} f(t,x,v)\,\ud x \ud v=0, \quad \frac{\ud}{\ud t}\int_{\R^{2d}} v\,f(t,x,v)\,\ud x \ud v=0.
		\]
		\item[\textup{(2)}] \textup{(Dissipation estimates)}:
		\[
		\frac{\ud}{\ud t}\int_{\R^{2d}} |v|^2\,f(t,x,v)\,\ud x \ud v\le 0, \quad \frac{\ud}{\ud t}H[f(t)]\ge 0, \quad \frac{\ud}{\ud t}\int_{\R^{2d}}|f(t,x,v)|^p\,\ud x \ud v\ge 0.
		\]
		\end{enumerate}
\end{proposition}

\begin{proof}
	Let $g=g(v)$ be a smooth test function. Then, we use the sufficient decay of $f$ at infinity and an integration by parts in $x$ and $v$ to obtain
	\begin{align}
	\begin{aligned} \label{B-1}
	& \frac{\ud}{\ud t}\int_{\R^{2d}} g(v)\,f(t,x,v)\,\ud x\ud v \\
	&\hspace{0.5cm}=\int_{\R^{2d}} g(v)\,\partial_t f(t,x,v)\,\ud x\ud v =\int_{\R^{2d}} g(v)\Bigl[-v\cdot\nabla_x f+\gamma\,\nabla_v\cdot\bigl(E[f]\,f\bigr)\Bigr]\,\ud x\ud v\\
	&\hspace{0.5cm}=-\gamma\int_{\R^{2d}} \nabla_v g(v)\cdot E[f](t,x,v)\,f(t,x,v)\,\ud x\ud v.
	\end{aligned}
	\end{align}
In the sequel, we choose $g$ appropriately to derive desired estimates. \newline

\noindent\textup{(1)} We take $1, v$ for $g$, respectively in \eqref{B-1} to find 
$$\frac{\ud}{\ud t}\int_{\R^{2d}} f(t,x,v)\,\ud x\ud v=0,$$
and
\begin{align*}
 \frac{\ud}{\ud t}\int_{\R^{2d}} v\,f(t,x,v)\,\ud x\ud v &=-\gamma\int_{\R^{2d}} E[f](t,x,v)\,f(t,x,v)\,\ud x\ud v \\
&=-\gamma\int_{\R^{3d}}\bigl(v-v_*\bigr)\,f(t,x,v)\,f(t,x,v_*)\,\ud x\ud v\ud v_*=0,
\end{align*}
where the last identity follows by interchanging $(v,v_*)$ and using anti-symmetry of $(v-v_*)$. \newline
	
\noindent\textup{(2)} Again, we take $g(v)=|v|^2$ to see $\nabla_v g=2v$, and substitute this into \eqref{B-1} to get 
	\begin{align*}
		 \frac{\ud}{\ud t}\int_{\R^{2d}} |v|^2\,f(t,x,v)\,\ud x\ud v & =-2\gamma\int_{\R^{2d}} v\cdot E[f](t,x,v)\,f(t,x,v)\,\ud x\ud v\\
		&=-2\gamma\int_{\R^{3d}} v\cdot (v-v_*)\,f(t,x,v)\,f(t,x,v_*)\,\ud x\ud v\ud v_*\\
		&=-\gamma\int_{\R^{3d}} |v-v_*|^2\,f(t,x,v)\,f(t,x,v_*)\,\ud x\ud v\ud v_* \le 0,
	\end{align*}
	where we use the identity $2\,v\cdot(v-v_*)=|v-v_*|^2+|v|^2-|v_*|^2$ and symmetry in $(v,v_*)$ in the last line.  For the second estimate, we use handy notations
\[ f=f(t,x,v) \quad \mbox{and} \quad f_*=f(t,x,v_*), \]
and use $\partial_t(f\log f)=(1+\log f)\partial_t f$ and integration by parts to get 
	\begin{align*}
		\frac{\ud}{\ud t}H[f(t)]
		&=\int_{\R^{2d}} (1+\log f)\,\partial_t f\,\ud x\ud v\\
		&=\int_{\R^{2d}} (1+\log f)\Bigl[-v\cdot\nabla_x f+\gamma\,\nabla_v\cdot\bigl(E[f]\,f\bigr)\Bigr]\,\ud x\ud v\\
		&=-\int_{\R^{2d}} v\cdot\nabla_x(f\log f)\,\ud x\ud v
		-\gamma\int_{\R^{2d}} \nabla_v f\cdot E[f]\,\ud x\ud v\\
		&=-\gamma\int_{\R^{3d}} \nabla_v f\cdot (v-v_*)\,f_*\,\ud x\ud v\ud v_*.
	\end{align*}
Now, we use $\nabla_v\cdot\bigl((v-v_*)f_*\bigr)=d\,f_*$ and the integration by parts in $v$ to find 
	\[
	\frac{\ud}{\ud t}H[f(t)]
	=\gamma\int_{\R^{3d}} f\,\nabla_v\cdot\bigl((v-v_*)f_* \bigr)\,\ud x\ud v\ud v_*
	=\gamma d\int_{\R^{3d}} f\,f_*\,\ud x\ud v\ud v_* \ge 0.
	\]
Next, we return to the last estimate. For $p\ge 1$, we use $\partial_t(f^p)=p f^{p-1}\partial_t f$ and  the integration by parts to find 
	\begin{align*}
		\frac{\ud}{\ud t}\int_{\R^{2d}} f(t,x,v)^p\,\ud x\ud v  &=p\int_{\R^{2d}} f^{p-1}\partial_t f\,\ud x\ud v\\
		&=-\int_{\R^{2d}} v\cdot\nabla_x(f^p)\,\ud x\ud v
		-\gamma p(p-1)\int_{\R^{2d}} f^{p-1}\nabla_v f\cdot E[f]\,\ud x\ud v\\
		&=-\gamma(p-1)\int_{\R^{2d}} \nabla_v(f^p)\cdot E[f]\,\ud x\ud v\\
		&=\gamma(p-1)\int_{\R^{2d}} f^p\,\nabla_v\cdot E[f]\,\ud x\ud v.
	\end{align*}
	Now, we use $E[f](t,x,v)=\int_{\R^d}(v-v_*)f(t,x,v_*)\,\ud v_*$ to get 
	\[  \nabla_v\cdot E[f]=\int_{\R^d}\nabla_v\cdot\bigl((v-v_*)f_*\bigr)\,\ud v_*
	=d\int_{\R^d} f_*\,\ud v_*. \]
	This yields the desired estimate:
	\[
	\frac{\ud}{\ud t}\int_{\R^{2d}} f^p\,\ud x\ud v
	=\gamma d(p-1)\int_{\R^{3d}} f(t,x,v)^p\,f(t,x,v_*)\,\ud x\ud v\ud v_* \ge 0.
	\]
\end{proof}

\begin{remark}[Entropy sign and comparison]\label{rem:entropy-sign}\textup{
	Proposition~\ref{evolution}(2) shows that
	\[
	H[f]:=\int_{\R^{2d}} f\log f\,\ud x\ud v
	\]
	is nondecreasing along smooth solutions of~\eqref{Eq}. Note that many authors instead call
	$S[f]:=-\int f\log f\,\ud x\ud v=-H[f]$ the (Boltzmann) entropy; in that convention our result reads that the
	physical entropy $S[f(t)]$ is nonincreasing. This is consistent with the order-forming nature of alignment:
	the Cucker--Smale dynamics reduces velocity fluctuations and promotes coherent motion. In contrast,
	$H[f]$ is formally conserved for the collisionless Vlasov--Poisson equation, while it is typically
	nonincreasing (equivalently, $S$ nondecreasing) for dissipative kinetic models such as
	Vlasov--Fokker--Planck and the Boltzmann/Landau equations.}
\end{remark}

\subsection{Spatially homogeneous equation}\label{sec:2.2}
In this subsection, we consider the spatially homogeneous setting, namely
\[
f(t,x,v)=f(t,v),
\]
so that there is no dependence on the spatial variable $x$. In this case, the equation \eqref{Eq} reduces to
\begin{equation}\label{shEq}
\partial_t f=\gamma\,\nabla_v\cdot\bigl(E_0[f]\,f\bigr), \quad t > 0,~~v \in {\mathbb R}^d, 
\end{equation}
where
\[
E_0[f](t,v)=\int_{\R^d}(v-v_*)\,f(t,v_*) \ud v_*.
\]
Moreover, the analogues of Proposition~\ref{evolution} remain valid for~\eqref{shEq}.
In particular, we use the conservation of mass and momentum to rewrite~\eqref{shEq} as a linear transport equation,
from which an explicit formula for the solution follows.

\begin{theorem}\label{shso}
	Let $f_0\in C^1(\R^d)$ be nonnegative, with finite mass and momentum:
	\[
	m:=\int_{\R^d} f_0(v)\ud v, \quad 
	p:=\int_{\R^d} v\,f_0(v) \ud v, \quad \overline v:=p/m.
	\]
	Then the solution to~\eqref{shEq} with initial datum $f(0,v)=f_0(v)$ is given by
	\begin{equation}\label{solution}
	    f(t,v)=\mathrm{e}^{\gamma mdt}f_0\left(\overline{v}+\mathrm{e}^{\gamma mt}(v-\overline{v})\right).
	\end{equation}
\end{theorem}
\begin{proof} 
By the same calculation with Proposition~\ref{evolution}, we can show that the mass and momentum are also preserved by the flow \eqref{shEq}:
\[ \frac{\mathrm{d}}{\mathrm{d}t}\left(\int_{\mathbb{R}^d}f(t,v)\mathrm{d}v\right)=0, \quad \frac{\mathrm{d}}{\mathrm{d}t}\left(\int_{\mathbb{R}^d}vf(t,v)\mathrm{d}v\right)=-\gamma\int_{\mathbb{R}^{2d}}(v-v_*)f(t,v)f(t,v_*)\mathrm{d}v\mathrm{d}v_*=0. \]
Thus, we have
$$\int_{\mathbb{R}^d}f(t,v)\mathrm{d}v=m,\quad\int_{\mathbb{R}^d}vf(t,v)\mathrm{d}v=p,\quad\forall~t\geq 0.$$
Then we have
$$E_0[f](t,v)=mv-p=m(v-\overline{v}),$$
and the equation \eqref{shEq} can be reduced  to
\begin{equation}\label{degenerate}
    \partial_t f=\gamma m(v-\overline{v})\cdot\nabla_v f+\gamma mdf.
\end{equation}
The characteristic curve of \eqref{degenerate} is given by
$$\frac{\mathrm{d}V}{\mathrm{d}t}=-\gamma m(V-\overline{v}),$$
whose solution with initial data $V(0)=v_0$ is
$$V(t)=\overline{v}+\mathrm{e}^{-\gamma mt}(v_0-\overline{v}).$$
Along the characteristic curve, we have
$$\frac{\mathrm{d}}{\mathrm{d}t}f(t,V(t))=\gamma mdf(t,V(t)).$$
This yields
$$f(t,V(t))=\mathrm{e}^{\gamma mdt}f(0,V(0))=\mathrm{e}^{\gamma mdt}f_0(v_0).$$
Therefore, the solution is given by the following explicit formula:
$$f(t,v)=\mathrm{e}^{\gamma mdt}f_0\left(\overline{v}+\mathrm{e}^{\gamma mt}(v-\overline{v})\right).$$
\end{proof}

The explicit formula in Theorem~\ref{shso} yields quantitative flocking estimates and precise growth/decay rates
for several basic functionals of the spatially homogeneous solution. Below, we set 
		\[
		R(t)=\sup\left\{|v-\overline{v}|\Big|~f(t,v)\neq 0\right\}.
		\]

\begin{corollary}\label{shcb}
	Under the same assumptions as in Theorem~\ref{shso}, the following assertions hold.
	\begin{enumerate}
		\item[\textup{(1)}]The $L^{\infty}$-norm of $f(t)$ grows exponentially fast:
		\[
		\Vert f(t)\Vert_{L^{\infty}}=\mathrm{e}^{\gamma mdt}\Vert f_0\Vert_{L^{\infty}}.
		\]
		\item[\textup{(2)}] The support of $f(t)$ shrinks exponentially fast towards  the singleton set $ \{ \overline{v} \}$:
		\[
		R(t)=\mathrm{e}^{-\gamma mt}R(0),
		\]
		\item[\textup{(3)}]The kinetic energy of the system converges exponentially fast to $m|\overline{v}|^2$:
		\[
		\int_{\mathbb{R}^d}|v|^2f(t,v)\ud v
		=m|\overline{v}|^2+\mathrm{e}^{-2\gamma mt}\int_{\mathbb{R}^d}|v-\overline{v}|^2f_0(v)\ud v.
		\]
		\item[\textup{(4)}]The entropy grows linearly:
		\[
		H[f(t)]=H[f_0]+\gamma mdt\int_{\mathbb{R}^d} f_0(v)\ud v, \quad t \geq 0.
		\]
	\end{enumerate}
\end{corollary}
\begin{proof}
The first two assertions follow directly from the explicit formula~\eqref{solution}.
	To verify last two assertions \textup{(3)} and \textup{(4)}, we substitute~\eqref{solution} and use the change of variables:
	\[
	w=\overline{v}+\mathrm{e}^{\gamma mt}(v-\overline{v})
	\qquad\text{so that}\qquad
	v=\overline{v}+\mathrm{e}^{-\gamma mt}(w-\overline{v}),\quad
	\ud v=\mathrm{e}^{-\gamma mdt}\ud w.
	\]
For the third assertion, we use~\eqref{solution} and the above change of variables to find 
	\begin{align*}
		\int_{\R^d}|v|^2 f(t,v)\ud v
		&=\mathrm{e}^{\gamma mdt}\int_{\R^d}|v|^2
		f_0\!\left(\overline{v}+\mathrm{e}^{\gamma mt}(v-\overline{v})\right)\ud v\\
		&=\int_{\R^d}\left|\overline{v}+\mathrm{e}^{-\gamma mt}(w-\overline{v})\right|^2 f_0(w)\ud w\\
		&=\int_{\R^d}\Bigl(|\overline{v}|^2
		+2\mathrm{e}^{-\gamma mt}\,\overline{v}\cdot(w-\overline{v})
		+\mathrm{e}^{-2\gamma mt}|w-\overline{v}|^2\Bigr) f_0(w)\ud w.
	\end{align*}
	Since $p=\int_{\R^d} w f_0(w)\ud w, \quad m=\int_{\R^d} f_0(w)\ud w$ and $\overline{v}=p/m$, we have
	\[
	\int_{\R^d}\overline{v}\cdot(w-\overline{v})\,f_0(w)\ud w
	=\overline{v}\cdot p-m|\overline{v}|^2=0.
	\]
	 Therefore, we have 
	\[
	\int_{\R^d}|v|^2 f(t,v)\ud v
	=m|\overline{v}|^2+\mathrm{e}^{-2\gamma mt}\int_{\R^d}|w-\overline{v}|^2 f_0(w)\ud w.
	\]
	For the last assertion, again by~\eqref{solution} and the same change of variables, we have
	\begin{align*}
		H[f(t)]
		&=\int_{\R^d} f(t,v)\log f(t,v)\ud v\\
		&=\int_{\R^d}\mathrm{e}^{\gamma mdt}f_0\!\left(\overline{v}+\mathrm{e}^{\gamma mt}(v-\overline{v})\right)
		\Bigl(\gamma mdt+\log f_0\!\left(\overline{v}+\mathrm{e}^{\gamma mt}(v-\overline{v})\right)\Bigr)\ud v\\
		&=\gamma mdt\int_{\R^d} f_0(w)\ud w+\int_{\R^d} f_0(w)\log f_0(w)\ud w\\
		&=H[f_0]+\gamma mdt\int_{\R^d} f_0(v)\ud v.
	\end{align*}
\end{proof}

\section{Spatially inhomogeneous model} \label{sec:3}
In this section, we return to the full equation~\eqref{Eq} and refer to it as the spatially inhomogeneous model,
in contrast with its spatially homogeneous reduction~\eqref{shEq} studied in the previous section. In the inhomogeneous
setting, one cannot expect an explicit solution formula as in~\eqref{solution}. Nevertheless, a priori estimates allow us
to establish a local well-posedness in suitable function spaces, as well as global well-posedness, when the interaction
strength $\gamma$ is sufficiently small. Moreover, whenever a global solution exists, we also derive an asymptotic
completeness result.

\subsection{Local well-posedness} \label{sec:3.1}
In this subsection, we establish a local well-posedness for inhomogeneous equation~\eqref{Eq} in the function spaces
$W^{1,\infty}$ and $C_b^{1,\alpha}$, respectively. We emphasize that, unlike the usual kinetic Cucker--Smale
or Vlasov-type models where the force field has a convolution structure in $x$, the local alignment term $E[f]$ in~\eqref{Eq}
is genuinely nonlocal only in $v$ and depends pointwise on $x$. This feature makes a well-posedness theory more delicate;
see Remark~\ref{rk:difficulty} for further discussion. Consider the Cauchy problem for the kinetic equation:
\begin{equation}\label{C-0}
\begin{cases}
\displaystyle \partial_t f +v\cdot\nabla_x f=\gamma\nabla_v\cdot\left(E[f]f\right), \quad t > 0,~x, v \in {\mathbb R}^d, \\
\displaystyle f(0,\cdot)=f_0.
\end{cases}
\end{equation}
In what follows, we briefly outline a proof strategy to establish the local existence following the steps below, which is a standard method for nonlinear partial differential equations. \newline

\noindent $\bullet$~Step A.1:~Given initial data $f_0$ in some proper function spaces, we find a complete metric space ${\mathcal B}$ such that for $f\in {\mathcal B}$, $E[f]$ is well-defined and the characteristic system:
$$\begin{cases}
    \dfrac{\mathrm{d}X}{\mathrm{d}t}=V,\vspace{0.2cm}\\
    \dfrac{\mathrm{d}V}{\mathrm{d}t}=-\gamma E[f](t,X,V),
\end{cases}$$
is well-posed. \newline

\noindent $\bullet$~Step A.2:~We show that for $f\in {\mathcal B}$, the linear Cauchy problem
\begin{equation}\label{linear}
\begin{cases}
    \partial_th+v\cdot\nabla_xh-\gamma \nabla_v\cdot(E[f]h)=0,\quad 0\leq t\leq T,\\
    h(0,x,v)=f_0(x,v)
\end{cases}
\end{equation}
has a unique solution $h\in {\mathcal B}$ for $T$ small enough, so that we can define the operator $\Gamma:{\mathcal B}\rightarrow {\mathcal B}$ by $\Gamma f=h$. Then $f$ solves \eqref{Eq} if and only if $f$ is a fixed point of $\Gamma$. \newline

\noindent $\bullet$~Step A.3:~Show that $\Gamma$ has some contraction properties and use the Picard iteration process to construct the solution. 
Finally, we check the regularity of this solution.

\vspace{0.5cm}
As a preparation, we denote the characteristic vector field by
\[
\Psi_f(t,z)=\begin{pmatrix}
    v\\
    -\gamma E[f](t,z)
\end{pmatrix}, 
\]
where  $z= (x, v) \in \R^{2d}$. 
Then the characteristic system associated with the linear Cauchy problem \eqref{linear} can be rewritten as
\begin{equation}\label{cheq}
    \frac{\mathrm{d}Z_f}{\mathrm{d}t}=\Psi_f(t,Z_f),
\end{equation}
where $Z_f=(X_f,V_f)$.
Suppose that $f\in C([0,T]\times\mathbb{R}^{2d})\cap L^{\infty}([0,T],W^{1,\infty}(\mathbb{R}^{2d}))$ with $v$-$\mathrm{supp}(f(t)) \subseteq B_R:=\{v\in\mathbb{R}^d|\ |v|\leq R\}$ for some $R>0$ and any $t\in[0,T]$. Then we can see that $\Psi_f$ is locally Lipschitz with respect to $z$. Thus the characteristic system \eqref{cheq} is well-posed. When the solution exists, we let $Z_f(s,t;z)$ denote the characteristic curve with initial data $Z_f(t,t;z)=z$, where $0\leq s\leq t\leq T$. Then along the characteristic curve $Z_f(s,t;z)$, we have
$$\frac{\mathrm{d}}{\mathrm{d}s}h\left(s,Z_f(s,t;z)\right)=\gamma d\rho_f\left(s,X_f(s,t;z)\right)h\left(s,Z_f(s,t;z)\right),$$
where the spatial density $\rho_f$ is defined as $\displaystyle\rho_f(t,x)=\int_{\mathbb{R}^d}f(t,x,v_*)\mathrm{d}v_*.$
Thus, we have
\begin{equation}\label{h}
h(t,z)=f_0(Z_f(0,t;z))\exp{\left(\gamma d\int_0^t\rho_f(s,X_f(s,t;z))\mathrm{d}s\right)}.
\end{equation}
Note that the relation \eqref{h} provides an explicit representation formula for the linear Cauchy problem~\eqref{linear}. This formula will be used repeatedly to derive the estimates needed in the well-posedness analysis. Next, we  state two local well-posedness results on local solutions.

\begin{theorem}\label{W1infty}
	Let $f_0\in W^{1,\infty}(\R^{2d})$ be nonnegative and satisfy
	\begin{equation} \label{C-0-1}
	v\textup{-}\mathrm{supp}(f_0)\subset B_{R_0}, \quad \mbox{for some $R_0>0$}.
	\end{equation}
	Then there exists $T>0$ and a nonnegative solution $f\in W^{1,\infty}\big([0,T]\times\R^{2d}\big)$
		to~\eqref{C-0} satisfying the boundedness of $v$-support:
	\begin{equation}\label{supp}
		v\textup{-}\mathrm{supp}\bigl(f(t)\bigr)\subset B_{R_0}, \quad t\in[0,T].
	\end{equation}
	The solution is unique in the class of $W^{1,\infty}\big([0,T]\times\R^{2d}\big)$ functions satisfying~\eqref{supp}.
\end{theorem}
\begin{proof}
Since the proof is rather lengthy, we leave its proof in Section \ref{sec:3.1.1}.
\end{proof}
The same conclusion holds in H\"older spaces $C_b^{1,\alpha}$.
\begin{theorem}\label{Cb1alpha}
For $\alpha\in(0,1],$ let $f_0\in C_b^{1,\alpha}(\R^{2d})$ be nonnegative satisfying
	\[
	v\textup{-}\mathrm{supp}(f_0)\subset B_{R_0}, \quad \mbox{for some $R_0>0$}.
	\]
	Then there exists $T>0$ and a nonnegative solution $f\in C^1\big([0,T]\times\R^{2d}\big)\cap L^{\infty}\big([0,T],C_b^{1,\alpha}(\R^{2d})\big)$ to~\eqref{C-0}  satisfying~\eqref{supp}. Moreover, the solution is unique in the class
	$C^1\big([0,T]\times\R^{2d}\big)\cap L^{\infty}\big([0,T],C_b^{1,\alpha}(\R^{2d})\big)$
	functions satisfying~\eqref{supp}.
\end{theorem}
\begin{proof}
Since the proof is rather lengthy, we leave its proof in Section \ref{sec:3.1.2}.
\end{proof}
We first notice that nonnegative $W^{1,\infty}$ solutions to \eqref{C-0} must satisfy \eqref{supp}.
\begin{lemma}\label{desupp1}
    Let $f_0\in W^{1,\infty}(\R^{2d})$ be nonnegative and satisfy
    \[
	v\textup{-}\mathrm{supp}(f_0)\subset B_{R_0}, \quad \mbox{for some $R_0>0$}.
	\]
    Assume that $f\in W^{1,\infty}\big([0,T]\times\R^{2d}\big)$ is a nonnegative solution to \eqref{C-0}, then we have
    $$v\textup{-}\mathrm{supp}\bigl(f(t)\bigr)\subset B_{R_0}, \quad t\in[0,T].$$
\end{lemma}
\begin{proof}
    Let $\Phi:[0,\infty)\to[0,\infty)$ be an arbitrary smooth, convex, nondecreasing funtion, with
	\[
	\Phi(r)=0 \ \text{for } r\le R_0,
	\qquad
	\Phi(r)>0 \ \text{for } r>R_0.
	\]
	Define
	\[
	M_\Phi(t):=\int_{\mathbb R^{2d}} \Phi(|v|)\,f(t,x,v)\,\ud x\,\ud v.
	\]
	Then we differentiate $M_{\Phi}$ using the equation \eqref{C-0}. The $x$-transport term integrates to $0$ on $\mathbb R^d_x$. For the $v$-divergence term,
	integration by parts in $v$ yields
	\[
	\frac{\ud}{\ud t}M_\Phi(t)
	=
	\gamma\int_{\mathbb{R}^{2d}} \Phi(|v|)\,\nabla_v\cdot(E f)\,\ud x\,\ud v
	=
	-\gamma\int_{\mathbb{R}^{2d}} \nabla_v\Phi(|v|)\cdot E[f]f\,\ud x\,\ud v.
	\]
	Since $\nabla_v\Phi(|v|)=\Phi'(|v|)\frac{v}{|v|}$ and $E[f]=v\rho-m=\rho(v-u)$, where $\rho=\rho_f$ and
    \[
    m(t,x)=\int_{\mathbb{R}^d}vf(t,x,v)\mathrm{d}v,\quad u(t,x)=\frac{m(t,x)}{\rho(t,x)},
    \]
    we get
	\[
	\frac{\ud}{\ud t}M_\Phi(t)
	=
	-\gamma\int_{\mathbb{R}^{2d}} \Phi'(|v|)\frac{v}{|v|}\cdot\big(v\rho-m\big)\, f\,\ud x\,\ud v.
	\]
	At fixed $(t,x)$, writing $d\mu_x(v):=\rho^{-1} f(t,x,v)\,dv$ (when $\rho>0$), the integrand becomes
	\[
	\rho^2 \int_{\mathbb{R}^d} \Phi'(|v|)\frac{v}{|v|}\cdot (v-u)\, \ud\mu_x(v).
	\]
	  Since $\Phi$ is nondecreasing and convex, $\Phi'$ is nonnegative and nondecreasing. Hence by Cauchy-Schwarz inequality and Chebyshev integral inequality, we have
      \begin{align*}
          \int_{\mathbb{R}^d} \Phi'(|v|)\frac{v}{|v|}\cdot (v-u)\, \ud\mu_x(v)&\geq\int_{\mathbb{R}^d} \Phi'(|v|)(|v|-|u|)\, \ud\mu_x(v)\\
          &\geq\left(\int_{\mathbb{R}^d} \Phi'(|v|)\, \ud\mu_x(v)\right)\cdot\left(\int_{\mathbb{R}^d}(|v|-|u|)\, \ud\mu_x(v)\right).
      \end{align*}
      Note that
      $$\int_{\mathbb{R}^d}(|v|-|u|)\, \ud\mu_x(v)\geq\left|\int_{\mathbb{R}^d}v\,\ud\mu_x(v)\right|-|u|=0.$$
      Therefore
	\[
	\frac{\ud}{\ud t}M_\Phi(t)\le 0.
	\]
	But $M_\Phi(0)=0$ because $\Phi(|v|)=0$ on the initial support, hence $M_\Phi(t)\equiv 0$ for all $t$.
	This forces $f(t,x,v)=0$ a.e.\ on $\{|v|>R_0\}$, i.e.\ the velocity support cannot leave $B_{R_0}$.
\end{proof}
\subsubsection{Proof of the existence part of Theorem \ref{W1infty}} \label{sec:3.1.1}
In this part, we perform a proof strategy which was delineated in the beginning of this section to derive an existence result for local $W^{1,\infty}$ solutions. \newline

Suppose that an initial datum $f_0\in W^{1,\infty}(\mathbb{R}^{2d})$ has a compact $v$-support:
\[ v\textup{-}\mathrm{supp}(f_0)\subseteq B_{R_0}, \quad \mbox{for some $R_0>0$.} \]
Next, we show that there exists some $T>0$ and a solution $f\in W^{1,\infty}([0,T]\times\mathbb{R}^{2d})$ to \eqref{C-0} in several steps. \newline

\noindent $\bullet$~Step B.1 (Definition of the metric space ${\mathcal B}$):~We set 
\[ M_0:=\Vert f_0\Vert_{W^{1,\infty}}<+\infty, \quad R:=2R_0>R_0 \quad \mbox{and} \quad M:=2M_0+1>M_0 \]
and define the set ${\mathcal B}$:
\[ {\mathcal B}:=\left\{f\in C([0,T]\times\mathbb{R}^{2d})|\ \Vert f(t)\Vert_{W^{1,\infty}} < M \quad \textup{and} \quad v\textup{-}\mathrm{supp}(f(t)) \subseteq B_R\textup{ for }t\in [0,T]\right\}, \]
where $T$ is some small number to be determined. Then the set ${\mathcal B}$ equipped with the distance
$$d(f,g):=\sup_{t\in[0,T]}\Vert f(t)-g(t)\Vert_{W^{1,\infty}},$$
becomes a complete metric space.
For $f\in {\mathcal B}$, according to the discussion before, the characteristic system \eqref{cheq} is well-posed and gives rise to the characteristic curve $Z_f(s,t;z)$. \newline

\noindent $\bullet$~Step B.2 (Definition of the operator $\Gamma$):  In the sequel, we set $C=C(\gamma ,d,R,M)$ to denote a general constant depending on $\gamma ,d,R,M$, which may vary from line to line. Given $f\in {\mathcal B}$, it is obvious that on $[0,T]\times\mathbb{R}^d\times B_R$, we have $$|\Psi_f(t,z)|\leq C.$$
Thus, as long as $T<\dfrac{R_0}{C}$, the characteristic curve starting from $\mathbb{R}^d\times B_{R_0}$ will not escape $\mathbb{R}^d\times B_R$ up to time $T$, hence the solution $h$ to \eqref{cheq} satisfies
$$v\textup{-}\mathrm{supp}(h(t)) \subseteq B_R,\quad\forall~ t\in[0,T].$$ Indeed, once we showed the existence of $W^{1, \infty}$ solutions,  we can choose  $R= R_0$ by  Lemma \ref{desupp}. Next, we need to estimate $\Vert h(t)\Vert_{W^{1,\infty}}$. Note that 
\[
\Vert\rho_f\Vert_{L^{\infty}}\leq C R^d \|f\|_{L^\infty} \leq C. 
\]
Hence, it follows from \eqref{h} that 
$$\sup_{[0,T]\times\mathbb{R}^{2d}}|h(t,z)|\leq\sup_{\mathbb{R}^{2d}}|f_0(z)|\mathrm{e}^{CT}.$$
It follows that for $T$ small enough, we have
$$\sup_{[0,T]\times\mathbb{R}^{2d}}|h(t,z)|\leq 2\sup_{\mathbb{R}^{2d}}|f_0(z)|.$$ To estimate $\mathrm{Lip}(h(t))$, we first need to estimate $\mathrm{Lip}(Z_f)$. From the expression of $\Psi_f$, we can easily get 
\[
\| \gamma E [f] \|_{W^\infty} \leq C R^{d+1} \|f\|_{W^{1, \infty}}.
\]
Thus, we have
$$\mathrm{Lip}(\Psi_f(t))\leq C,\quad\forall~ t\in[0,T].$$
Note that $Z_f(s,t;z)$ satisfies
$$Z_f(s,t;z)=z-\int_s^t \Psi_f(\tau,Z_f(\tau,t;z))\mathrm{d}\tau,$$
which implies 
$$|Z_f(s,t;z_1)-Z_f(s,t;z_2)|\leq|z_1-z_2|+C\int_s^t|Z_f(\tau,t;z_1)-Z_f(\tau,t;z_2)|\mathrm{d}\tau.$$
By Gronwall's inequality (for a differential inequality backward in time), we get
$$|Z_f(s,t;z_1)-Z_f(s,t;z_2)|\leq|z_1-z_2|\exp{(C(t-s))}\leq\mathrm{e}^{CT}|z_1-z_2|.$$
We also note that $\rho_f \in W^{1, \infty}$ with 
\[
\| \rho_f\|_{W^{1, \infty}} \leq C  R^d \|f\|_{W^{1, \infty}}.   
\]
Therefore, it follows from \eqref{h} and the Mean Value Theorem: 
\[ |e^{x} - e^y| \leq e^{\max\{|x|, |y|\}} |x- y| \]
to get 
\begin{align*}
 &|h(t,z_1)-h(t,z_2)|\\
 & \hspace{1cm} \leq \Big| f_0(Z_f(0,t;z_1))\exp{\Big (\gamma d\int_0^t\rho_f(s,X_f(s,t;z_1))\mathrm{d}s \Big)} \\
 & \hspace{1.5cm} -f_0(Z_f(0,t;z_2))\exp{\Big(\gamma d\int_0^t\rho_f(s,X_f(s,t;z_1))\mathrm{d}s\Big)} \Big |  \\
 & \hspace{1cm} +\Big |f_0(Z_f(0,t;z_2)) \exp \Big(\gamma d\int_0^t\rho_f(s,X_f(s,t;z_1))\mathrm{d}s\Big)  \\
 & \hspace{1.5cm} - f_0(Z_f(0,t;z_2))\exp{\Big(\gamma d\int_0^t\rho_f(s,X_f(s,t;z_2))\mathrm{d}s \Big)} \Big|\\
 & \hspace{1cm} \leq \mathrm{e}^{CT}\mathrm{Lip}(f_0)|z_1-z_2|+CT\mathrm{e}^{CT}|z_1-z_2|.
\end{align*}
Thus for $T$ small enough, we have
$$\mathrm{Lip}(h(t))\leq 2\mathrm{Lip}(f_0)+1.$$
Hence, we have
\[ \Vert h(t)\Vert_{W^{1,\infty}}\leq M, \quad h\in {\mathcal B}. \]
Therefore, we can define the operator $\Gamma: {\mathcal B}~\rightarrow~{\mathcal B}$ by $\Gamma f=h$. \newline

\noindent $\bullet$~Step B.3 (Constructing the solution):~Next, we want to show that $\Gamma$ has some contraction properties, which will be essential when constructing the solution. Actually, it is enough to show that $\Gamma$ is a contraction in the space $C^0$. Given $f,g\in {\mathcal B}$, we first need to estimate
\[ \Psi_f-\Psi_g \quad \mbox{and} \quad Z_f-Z_g. \]
Note that we also have
\[ f(0, \cdot) = g(0, \cdot)= f_0(\cdot), \quad \Psi_f(t,z)-\Psi_g(t,z)=\begin{pmatrix}
    0\\
    -\gamma\int_{\mathbb{R}^d}(v-v_*)(f(t,x,v_*)-g(t,x,v_*))\mathrm{d}v_*
\end{pmatrix}. 
\]
Thus, we have
\[ \left|\Psi_f(t,z)-\Psi_g(t,z)\right|\leq C\Vert f-g\Vert_{C^0}. \]
It follows that
\begin{align*}
\begin{aligned}
&\left|Z_f(s,t;z)-Z_g(s,t;z)\right|\\
& \hspace{1cm} =\left|\int_s^t\left(\Psi_f(\tau,Z_f(\tau,t;z))-\Psi_g(\tau,Z_g(\tau,t;z))\right)\mathrm{d}\tau\right|\\
& \hspace{1cm} \leq \int_s^t\left|\Psi_f(\tau,Z_f(\tau,t;z))-\Psi_f(\tau,Z_g(\tau,t;z))\right|\mathrm{d}\tau \\
& \hspace{1.2cm} +\int_s^t\left|\Psi_f(\tau,Z_g(\tau,t;z))-\Psi_g(\tau,Z_g(\tau,t;z))\right|\mathrm{d}\tau\\
& \hspace{1cm} \leq \int_s^t\mathrm{Lip}(\Psi_f(\tau))\cdot\left|Z_f(\tau,t;z)-Z_g(\tau,t;z)\right|\mathrm{d}\tau+C(t-s)\Vert f-g\Vert_{C^0}\\
& \hspace{1cm} \leq CT\Vert f-g\Vert_{C^0}+C\int_s^t\left|Z_f(\tau,t;z)-Z_g(\tau,t;z)\right|\mathrm{d}\tau.
\end{aligned}
\end{align*}
By Gronwall's lemma, we have
$$\left|Z_f(s,t;z)-Z_g(s,t;z)\right|\leq CT\Vert f-g\Vert_{C^0}\cdot\exp{\left(C(t-s)\right)}\leq CT\mathrm{e}^{CT}\Vert f-g\Vert_{C^0}.$$
It follows from \eqref{h} and the Mean Value Theorem that 
\begin{align}\label{contraction}
\begin{aligned}
&\Big |(\Gamma f)(t,z)-(\Gamma g)(t,z) \Big |\\
& \hspace{1cm} = \Big | f_0(Z_f(0,t;z))\exp{\Big(\gamma d\int_0^t\rho_f(s,X_f(s,t;z))\mathrm{d}s \Big)} \\
& \hspace{2cm} -f_0(Z_g(0,t;z))\exp \Big(\gamma d\int_0^t\rho_g(s,X_g(s,t;z))\mathrm{d}s \Big) \Big |\\
& \hspace{1cm} \leq \Big |f_0(Z_f(0,t;z))\exp \Big(\gamma d\int_0^t \rho_f(s,X_f(s,t;z))\mathrm{d}s \Big) \\
& \hspace{2cm}  - f_0(Z_f(0,t;z))\exp \Big(\gamma d\int_0^t\rho_f(s,X_g(s,t;z))\mathrm{d}s \Big) \Big|\\
& \hspace{1cm} + \Big|f_0(Z_f(0,t;z))\exp \Big (\gamma d\int_0^t\rho_f(s,X_g(s,t;z))\mathrm{d}s \Big) \\
& \hspace{2cm} -f_0(Z_f(0,t;z))\exp \Big (\gamma d\int_0^t\rho_g(s,X_g(s,t;z))\mathrm{d}s \Big) \Big |\\
& \hspace{1cm} +\Big| f_0(Z_f(0,t;z))\exp \Big(\gamma d\int_0^t\rho_g(s,X_g(s,t;z))\mathrm{d}s \Big) \\
& \hspace{2cm} -f_0(Z_g(0,t;z))\exp \Big (\gamma d\int_0^t\rho_g(s,X_g(s,t;z))\mathrm{d}s \Big) \Big |\\
& \hspace{1cm} \leq CT\Vert X_f-X_g\Vert_{C^0}+CT\Vert\rho_f-\rho_g\Vert_{C^0}+C\Vert Z_f-Z_g\Vert_{C^0}\\
& \hspace{1cm} \leq CT\mathrm{e}^{CT}\Vert f-g\Vert_{C^0}.
\end{aligned}
\end{align}
Thus, we have
\[ \Vert\Gamma f-\Gamma g\Vert_{C^0}\leq\frac{1}{2}\Vert f-g\Vert_{C^0}, \quad 0 < T \ll 1. \]
We define
\[ f^{(0)}(t,z)=f_0(z), \quad f^{(n)}=\Gamma^nf^{(0)}, \quad n \geq 1, \quad  \mbox{for $(t,z)\in[0,T]\times\mathbb{R}^{2d}$}, \]
Then we have
$$\Vert f^{(n+1)}-f^{(n)}\Vert_{C^0}\leq\frac{1}{2}\Vert f^{(n)}-f^{(n-1)}\Vert_{C^0}.$$
Thus $\{f^{(n)}\}$ is a Cauchy sequence in $C^0$ and converges to some function $f\in C([0,T]\times\mathbb{R}^{2d})$ with $v$-support contained in $B_R$. Besides, for a fixed $t\in[0,T]$, we also have
$$|f^{(n)}(t,z_1)-f^{(n)}(t,z_2)|\leq\mathrm{Lip}(f^{(n)}(t))|z_1-z_2|,\quad\forall z_1,z_2\in\mathbb{R}^{2d}.$$
Letting $n\rightarrow\infty$, we get
\begin{align*}
\begin{aligned}
    \Vert f(t)\Vert_{W^{1,\infty}}&=\Vert f(t)\Vert_{L^{\infty}}+\mathrm{Lip}(f(t))\\
    &\leq\limsup_{n\rightarrow\infty}\left(\Vert f^{(n)}(t)\Vert_{L^{\infty}}+\mathrm{Lip}(f^{(n)}(t))\right) \\
    &=\limsup_{n\rightarrow\infty}\Vert f^{(n)}(t)\Vert_{W^{1,\infty}}\leq M.
\end{aligned}
\end{align*}
This implies $f\in {\mathcal B}$ and 
$$\Vert\Gamma f-f\Vert_{C^0}\leq\Vert\Gamma f-\Gamma f^{(n)}\Vert_{C^0}+\Vert f^{(n+1)}-f\Vert_{C^0}\leq\frac{1}{2}\Vert f-f^{(n)}\Vert_{C^0}+\Vert f^{(n+1)}-f\Vert_{C^0}\rightarrow 0\quad (n\rightarrow\infty).$$
Hence we get $\Gamma f=f$, i.e. $f$ is a solution to \eqref{Eq}. Moreover, by \eqref{h} we know that $\partial_t f$ is also in $L^\infty$, which implies $f\in W^{1,\infty}([0,T]\times\mathbb{R}^{2d})$. The property \eqref{supp} follows from Lemma ~\ref{desupp1}.

\subsubsection{Proof of the existence part of Theorem \ref{Cb1alpha}} \label{sec:3.1.2}
By the same procedure as in Section \ref{sec:3.1.1}, we can derive the local existence for $C_b^{1,\alpha}$ solutions. \newline

Suppose that the initial datum $f_0\in C_b^{1,\alpha}(\mathbb{R}^{2d})$ has a compact $v$-support for some $\alpha\in(0,1]$. Then, we claim that there exists some $T>0$ and a solution $f\in C^1([0,T]\times\mathbb{R}^{2d})\cap L^{\infty}([0,T],C_b^{1,\alpha}(\mathbb{R}^{2d}))$ to \eqref{C-0} satisfying \eqref{supp}  for some $R>0$, with initial datum $f_0$. \newline

\noindent $\bullet$~Step C.1 (Definition of the metric space ${\mathcal B}$):~We set 
\[ M_0:=\Vert f_0\Vert_{C^{1,\alpha}}<+\infty, \quad R:=2R_0 \quad \mbox{and} \quad M:=2M_0+1, \]
and define
\[ {\mathcal B}:=\left\{f\in C([0,T]\times\mathbb{R}^{2d})|\ \Vert f(t)\Vert_{C^{1,\alpha}} < M \quad  \textup{and} \quad v\textup{-}\mathrm{supp}f(t)\subseteq B_R\textup{ for }t\in [0,T]\right\}, \]
where $T$ is a positive number to be determined. Then equipped with the distance
$$d(f,g):=\sup_{t\in[0,T]}\Vert f(t)-g(t)\Vert_{C^{1,\alpha}},$$
the new ${\mathcal B}$ becomes a complete metric space.

\smallskip 

\noindent $\bullet$~Step C.2 (Definition of the operator $\Gamma$): In the sequel, let $C=C(\gamma,d,R,M,\alpha)$ denote a general constant depending on $\gamma,d,R,M,\alpha$, which may vary from line to line. Since $C_b^{1,\alpha}$ is a subspace of $W^{1,\infty}$, the arguments in the previous part  still apply here. That means, given $f\in {\mathcal B}$, the characteristic equation \eqref{cheq} is well-posed and gives rise to the characteristic curve $Z_f(s,t;z)$. The linear Cauchy problem \eqref{linear} is solved by \eqref{h}:
$$h(t,z)=f_0(Z_f(0,t;z))\exp{\left(\gamma d\int_0^t\rho_f(s,X_f(s,t;z))\mathrm{d}s\right)}.$$
Besides, for $T$ small enough, we have
\[ v\textup{-}\mathrm{supp}(h(t))\subseteq B_R,\quad\forall~t\in[0,T], \quad \mbox{and} \quad \sup_{[0,T]\times\mathbb{R}^{2d}}|h(t,z)|\leq 2\sup_{\mathbb{R}^{2d}}|f_0(z)|. \]
To estimate $\nabla_z h$, we first need to estimate the Jacobian of $Z_f$ with respect to $z$, denoted by $\dfrac{\partial Z_f}{\partial z}(s,t;z)$. Note that $Z_f(s,t;z)$ satisfies
$$Z_f(s,t;z)=z-\int_s^t \Psi_f(\tau,Z_f(\tau,t;z))\mathrm{d}\tau.$$
Thus, we have
\begin{equation}\label{dz}
    \frac{\partial Z_f}{\partial z}(s,t;z)=I_{2d}-\int_s^t\frac{\partial\Psi_f}{\partial z}(\tau,Z_f(\tau,t;z))\cdot\frac{\partial Z_f}{\partial z}(\tau,t;z)\mathrm{d}\tau.
\end{equation}
From the expression of $\Psi_f$ we can get
$$\frac{\partial\Psi_f}{\partial z}(t,z)=\begin{pmatrix}
    0& I_d\\
    -\gamma\int_{\mathbb{R}^d}(v-v_*)\otimes\nabla_xf(t,x,v_*)\mathrm{d}v_*& -\gamma\rho_f(t,x)I_d
\end{pmatrix}.$$
Then we have
$$\left\Vert\frac{\partial\Psi_f}{\partial z}(t)\right\Vert_{C^{\alpha}}\leq C.$$
Thus, it follows from \eqref{dz} that 
$$\left\Vert\frac{\partial Z_f}{\partial z}(s,t; z )\right\Vert_{L^{\infty}}\leq 1+C\int_s^t\left\Vert\frac{\partial Z_f}{\partial z}(s,t; z )\right\Vert_{L^{\infty}}\mathrm{d}\tau.$$
By Gronwall's inequality, we can obtain
$$\left\Vert\frac{\partial Z_f}{\partial z}(s,t;z)\right\Vert_{L^{\infty}}\leq\mathrm{e}^{C(t-s)}.$$
Moreover,
\begin{align*}
\begin{aligned}
 &\left|\frac{\partial Z_f}{\partial z}(s,t;z_1)-\frac{\partial Z_f}{\partial z}(s,t;z_2)\right|\\
  & \hspace{0.5cm}  \leq \int_s^t\left|\frac{\partial\Psi_f}{\partial z}(\tau,Z_f(\tau,t;z_1))\cdot\frac{\partial Z_f}{\partial z}(\tau,t;z_1)-\frac{\partial\Psi_f}{\partial z}(\tau,Z_f(\tau,t;z_2))\cdot\frac{\partial Z_f}{\partial z}(\tau,t;z_2)\right|\mathrm{d}\tau\\
  &  \hspace{0.5cm}   \leq \int_s^t\left|\frac{\partial\Psi_f}{\partial z}(\tau,Z_f(\tau,t;z_1))-\frac{\partial\Psi_f}{\partial z}(\tau,Z_f(\tau,t;z_2))\right|\cdot\left|\frac{\partial Z_f}{\partial z}(\tau,t;z_1)\right|\mathrm{d}\tau+\\
  &  \hspace{0.5cm} \int_s^t\left|\frac{\partial\Psi_f}{\partial z}(\tau,Z_f(\tau,t;z_2))\right|\cdot\left|\frac{\partial Z_f}{\partial z}(\tau,t;z_1)-\frac{\partial Z_f}{\partial z}(\tau,t;z_2)\right|\mathrm{d}\tau\\
  & \hspace{0.5cm}   \leq C(t-s)\mathrm{e}^{C(t-s)}|z_1-z_2|^{\alpha}+C\int_s^t\left|\frac{\partial Z_f}{\partial z}(\tau,t;z_1)-\frac{\partial Z_f}{\partial z}(\tau,t;z_2)\right|\mathrm{d}\tau.
\end{aligned}
\end{align*}
Using Gronwall's inequality again, we obtain
$$\left[\frac{\partial Z_f}{\partial z}(s,t;z)\right]_{\alpha}\leq C(t-s)\mathrm{e}^{C(t-s)}.$$
Then we can use \eqref{h} to get
\begin{align}
\begin{aligned} \label{C-0-2}
 \nabla_zh(t,z) &=\exp{\left(\gamma d\int_0^t\rho_f(s,X_f(s,t;z))\mathrm{d}s\right)} \\
 &\times \Big[ \frac{\partial Z_f}{\partial z}(0,t;z)\cdot\nabla_zf_0(Z_f(0,t;z)) \\
 & +\gamma df_0(Z_f(0,t;z))\int_0^t\frac{\partial Z_f}{\partial z}(s,t;z)\cdot\begin{pmatrix}
    \nabla_x\rho_f(s,X_f(s,t;z)) \\
    0
\end{pmatrix}\mathrm{d}s \Big ].
\end{aligned}
\end{align}
By using the fundamental inequality:
\[ [\varphi\psi]_{\alpha}\leq[\varphi]_{\alpha}\Vert\psi\Vert_{C^0}+\Vert\varphi\Vert_{C^0}[\psi]_{\alpha} \]
repeatedly, we see that 
$$\sup_{t\in[0,T]}\Vert\nabla_zh(t)\Vert_{C^{\alpha}}\leq 2\Vert\nabla_zf_0\Vert_{C^{\alpha}}+1$$
for $T$ small enough. Thus, we have
\[ \Vert h(t)\Vert_{C^{1,\alpha}}\leq M, \quad \mbox{i.e.,} \quad h\in {\mathcal B}. \]
Therefore, we can define the operator $\Gamma: {\mathcal B} \rightarrow {\mathcal B}$ by $\Gamma f=h$. \newline

\noindent $\bullet$~Step C.3 (Constructing the solution):~As in Case B.3, we still have 
$$\Vert\Gamma f-\Gamma g\Vert_{C^0}\leq\frac{1}{2}\Vert f-g\Vert_{C^0}$$
for $T$ small enough. We set 
\[ f^{(0)}(t,z)=f_0(z), \quad f^{(n)}=\Gamma^nf^{(0)} \quad \mbox{for $n\geq 1$ \quad for $(t,z)\in[0,T]\times\mathbb{R}^{2d}$}, \]
Then we have
$$\Vert f^{(n+1)}-f^{(n)}\Vert_{C^0}\leq\frac{1}{2}\Vert f^{(n)}-f^{(n-1)}\Vert_{C^0}, \quad n \geq 1.$$
Thus $\{f^{(n)}\}$ is a Cauchy sequence in $C^0$ and converges to some function $f\in C([0,T]\times\mathbb{R}^{2d})$ with $v$-support contained in $B_R$. Moreover, for a fixed $t\in[0,T]$, since $\Vert f^{(n)}(t)\Vert_{C^{1,\alpha}}\leq M$, by the Ascoli-Arzel\`{a} Theorem, there exists a subsequence $\{f^{(n_k)}(t)\}$ converging to $f(t)$ in $C^1_{\mathrm{loc}}(\mathbb{R}^{2d})$, and we have 
$$|\nabla f^{(n)}(t,z_1)-\nabla f^{(n)}(t,z_2)|\leq [\nabla f^{(n)}(t)]_{\alpha}|z_1-z_2|^{\alpha},\quad\forall z_1,z_2\in\mathbb{R}^{2d}.$$
Letting $n\rightarrow\infty$, we get
$$\Vert f(t)\Vert_{C^{1,\alpha}}\leq\limsup_{k\rightarrow\infty}\Vert f^{(n_k)}(t)\Vert_{C^{1,\alpha}}\leq M,$$
implying $f\in {\mathcal B}$. Thus, we have
$$\Vert\Gamma f-f\Vert_{C^0}\leq\Vert\Gamma f-\Gamma f^{(n)}\Vert_{C^0}+\Vert f^{(n+1)}-f\Vert_{C^0}\leq\frac{1}{2}\Vert f-f^{(n)}\Vert_{C^0}+\Vert f^{(n+1)}-f\Vert_{C^0}\rightarrow 0\quad (n\rightarrow\infty).$$
Hence we get 
\[ \Gamma f=f, \]
i.e. $f$ is a classical solution to \eqref{Eq}. The property \eqref{supp} follows from Lemma ~\ref{desupp1}. This completes the proof.

\begin{remark}\label{rk:difficulty}
    \textup{We cannot obtain the local existence for $C^1$ solutions using the methods above. The reason is, while getting the contraction estimate \eqref{contraction}, although it is a contraction in $C^0$, the constant $C$ depends on the uniform bound $M$ for the Lipschitz constant, which can be seen as the $C^1$-norm of the functions in ${\mathcal B}$. Therefore, it is predictable that when we want to get a similar contraction estimate in $C^1$, we need a uniform bound for the $C^2$-norm of functions in ${\mathcal B}$, which we do not have. For $W^{1,\infty}$ and $C_b^{1,\alpha}$ solutions, we use Lipschitz and Ascoli-Arzel\`{a} arguments respectively to avoid that estimate.}
\end{remark}
In the sequel, we present the uniqueness of $W^{1,\infty}([0,T]\times\mathbb{R}^{2d})$-solution. 
\begin{lemma}\label{estimate}
  Suppose that $f_0\in W^{1,\infty}(\mathbb{R}^{2d})$ is an initial datum satisfying $v$-$\mathrm{supp}(f_0)\subseteq B_{R_0}$, and let $f\in W^{1,\infty}([0,T]\times\mathbb{R}^{2d})$ be a solution to \eqref{C-0} satisfying \eqref{supp}. Then we have
    \begin{equation}\label{W}
        \Vert f(t)\Vert_{W^{1,\infty}}\leq\frac{\Vert f_0\Vert_{W^{1,\infty}}}{1-Ct(1+\Vert f_0\Vert_{W^{1,\infty}})},
    \end{equation}
    where $C$ is some constant depending on $\gamma,d$ and $R_0$.
\end{lemma}
\begin{proof} 
Since $f$ is a solution to \eqref{Eq} with initial datum $f_0$, we have
\begin{equation}\label{rep_f}
f(t,z)=f_0(Z_f(0,t;z))\exp{\left(\gamma d\int_0^t\rho_f(s,X_f(s,t;z))\mathrm{d}s\right)}.
\end{equation}
Thus, we have
\begin{equation}\label{sup}
   \Vert f(t)\Vert_{L^{\infty}}\leq\Vert f_0\Vert_{L^{\infty}}\cdot\exp{\left(C\int_0^t\Vert f(s)\Vert_{L^{\infty}}\mathrm{d}s\right)}. 
\end{equation}
It follows from the expression of $\Psi_f$ that
$$\mathrm{Lip}(\Psi_f(t))\leq C(1+\Vert f(t)\Vert_{W^{1,\infty}}). $$
Hence, we have
\begin{align*}
\begin{aligned}
& |Z_f(s,t;z_1)-Z_f(s,t;z_2)| \\
& \hspace{1cm} \leq|z_1-z_2|+\int_s^t|\Psi_f(\tau,Z_f(\tau,t;z_1))-\Psi_f(\tau,Z_f(\tau,t;z_2))|\mathrm{d}\tau\\
& \hspace{1cm} \leq|z_1-z_2|+C\int_s^t(1+\Vert f(\tau)\Vert_{W^{1,\infty}})|Z_f(\tau,t;z_1)-Z_f(\tau,t;z_2)|\mathrm{d}\tau.
\end{aligned}
\end{align*}
By Gronwall's lemma, we get
\begin{equation}\label{lip1}
    \mathrm{Lip}(Z_f(s,t))\leq\exp{\left(C\int_s^t(1+\Vert f(\tau)\Vert_{W^{1,\infty}})\mathrm{d}\tau\right)}.
\end{equation}
From \eqref{rep_f}, \eqref{lip1} and the Mean Value Theorem, we have
\begin{align*}
&|f(t,z_1)-f(t,z_2)| \leq |f_0(Z_f(0,t;z_1))-f_0(Z_f(0,t;z_2))|\exp{\left(\gamma d\int_0^t\rho_f(s,X_f(s,t;z_1))\mathrm{d}s\right)} \\
& \hspace{0.2cm} + |f_0(Z_f(0,t;z_2))|\cdot\left|\exp{\left(\gamma d\int_0^t\rho_f(s,X_f(s,t;z_1))\mathrm{d}s\right)}-\exp{\left(\gamma d\int_0^t\rho_f(s,X_f(s,t;z_2))\mathrm{d}s\right)}\right|\\
&\hspace{0.2cm} =\uppercase\expandafter{\romannumeral1}+\uppercase\expandafter{\romannumeral2}.
\end{align*}
Then we estimate the two terms respectively:
\begin{align*}
    \uppercase\expandafter{\romannumeral1}\leq& \mathrm{Lip}(f_0)\cdot\exp{\left(C\int_0^t(1+\Vert f(s)\Vert_{W^{1,\infty}})\mathrm{d}s\right)}|z_1-z_2|,\\
    \uppercase\expandafter{\romannumeral2}\leq&\Vert f_0\Vert_{L^{\infty}}\cdot\exp{\left(C\int_0^t\Vert f(s)\Vert_{L^{\infty}}\mathrm{d}s\right)}\\
    &\times C\int_0^t\mathrm{Lip}(f(s))\cdot\exp{\left(C\int_s^t(1+\Vert f(\tau)\Vert_{W^{1,\infty}})\mathrm{d}\tau\right)}\mathrm{d}s\cdot|z_1-z_2|\\
    \leq&\Vert f_0\Vert_{L^{\infty}}\cdot\exp{\left(C\int_0^t\Vert f(s)\Vert_{L^{\infty}}\mathrm{d}s\right)}\\
    &\times\int_0^t C(1+\Vert f(s)\Vert_{W^{1,\infty}})\cdot\exp{\left(C\int_s^t(1+\Vert f(\tau)\Vert_{W^{1,\infty}})\mathrm{d}\tau\right)}\mathrm{d}s\cdot|z_1-z_2|\\
    =&\Vert f_0\Vert_{L^{\infty}}\cdot\exp{\left(C\int_0^t\Vert f(s)\Vert_{L^{\infty}}\mathrm{d}s\right)}\cdot\left(\exp{\left(C\int_0^t(1+\Vert f(s)\Vert_{W^{1,\infty}})\mathrm{d}s\right)}-1\right)|z_1-z_2|.
\end{align*}
Thus, we have
\begin{align}\label{lip2}
\begin{aligned}
& \mathrm{Lip}(f(t)) \leq\mathrm{Lip}(f_0)\cdot\exp{\left(C\int_0^t(1+\Vert f(s)\Vert_{W^{1,\infty}})\mathrm{d}s\right)} \\
& \hspace{0.2cm} + \Vert f_0\Vert_{L^{\infty}}\cdot\exp{\left(C\int_0^t\Vert f(s)\Vert_{L^{\infty}}\mathrm{d}s\right)}\cdot\left(\exp{\left(C\int_0^t(1+\Vert f(s)\Vert_{W^{1,\infty}})\mathrm{d}s\right)}-1\right).
\end{aligned}
\end{align}
We combine \eqref{sup} and \eqref{lip2} to get the integral form of Gronwall-type inequality:
$$\Vert f(t)\Vert_{W^{1,\infty}}\leq\Vert f_0\Vert_{W^{1,\infty}}\exp{\left(C\int_0^t(1+\Vert f(s)\Vert_{W^{1,\infty}})\mathrm{d}s\right)}.$$
This yields
$$\Vert f(t)\Vert_{W^{1,\infty}}\leq\frac{\Vert f_0\Vert_{W^{1,\infty}}}{1-Ct(1+\Vert f_0\Vert_{W^{1,\infty}})}.$$
\end{proof} 
Now, we are ready to provide the uniqueness result for our equation \eqref{Eq}. 
\begin{proposition}\label{uniqueness}
Suppose that the initial datum $f_0\in W^{1,\infty}(\mathbb{R}^{2d})$ satisfies $v$-$\mathrm{supp}(f_0)\subseteq B_{R_0}$, and let $f,g\in W^{1,\infty}([0,T]\times\mathbb{R}^{2d})$ be two solutions to \eqref{C-0} with the same initial datum $f_0$. Furthermore suppose that $f$ and $g$ both satisfy the assumptions in Lemma \textup{~\ref{estimate}}. Then we have
\[  f \equiv g. \]
\end{proposition}
\begin{proof} 
We define
\[ t_0:=\sup \Big \{t\in[0,T]|~f(s)=g(s),\forall s\in[0,t] \Big \}. \]
Then, we have 
\[ f(t_0)=g(t_0)\in W^{1,\infty}(\mathbb{R}^{2d}) \quad \mbox{with $v$-support contained in $B_R$}. \]
Assume that $t_0<T$. Then by Proposition~\ref{estimate}, for $\varepsilon$ small enough, $\tilde{f}:=f|_{[t_0,t_0+\varepsilon]}$ and $\tilde{g}:=g|_{[t_0,t_0+\varepsilon]}$ belong to the space ${\mathcal B}$ in Section \ref{sec:3.1.1}~(with $R$ and $M$ modified a bit). Thus
$$\Vert\tilde{f}-\tilde{g}\Vert_{C^0}=\Vert\Gamma\tilde{f}-\Gamma\tilde{g}\Vert_{C^0}\leq\frac{1}{2}\Vert\tilde{f}-\tilde{g}\Vert_{C^0},$$
implying that $\tilde{f}=\tilde{g}$, which contradicts the definition of $t_0$. Therefore, we must have
$$f(t)=g(t),\quad\forall t\in[0,T].$$
\end{proof} 
Finally, we combine materials in Section \ref{sec:3.1.1} and Section \ref{sec:3.1.2} together with Proposition \ref{uniqueness} respectively to conclude the proofs for Theorem~\ref{W1infty} and Theorem~\ref{Cb1alpha}.

\subsection{Global well-posedness} \label{sec:3.2}
In this subsection, we consider the case when the initial data $f_0$ is nonnegative and compactly supported in both $x$ and $v$ variables.
After showing (i) the control of support, (ii) differential inequalities for norm of $f$ and gradients of $f$, we present a global well-posedness result for \eqref{C-0}, provided that the interaction constant $\gamma$ is small enough.
\begin{theorem}\label{global2}
For $d\geq 2$ and $T \in (0, \infty)$, let $f_0\in W^{1,\infty}(\mathbb{R}^{2d})$ be a nonnegative function with $\mathrm{supp}(f_0)\subseteq B_{R_0}\times B_{R_0}$ for some $R_0>0$. Then, the following assertions hold.
\begin{enumerate}
\item
There exists a positive constant $\gamma_0$ depending on $d, R_0$ and $\Vert f_0\Vert_{W^{1,\infty}}$, such that for $\gamma<\gamma_0$, there is a global nonnegative solution $f \in W^{1,\infty}([0,T]\times\mathbb{R}^{2d})$ to \eqref{C-0} satisfying
      \begin{equation}\label{2}
        v\textup{-}\mathrm{supp}(f(t)) \subseteq B_{R(t)},\quad\forall~t \in [0, T),
    \end{equation}
    for some nonincreasing function $R:[0,+\infty)\rightarrow[0,+\infty)$. 
    \vspace{0.1cm}
 \item   
The solution in the first assertion is unique among the family of solutions satisfying \eqref{2}.
\end{enumerate}
\end{theorem}
\begin{proof}
Since the proof is very lengthy and need several preparatory estimates, we leave its proof in Section \ref{sec:3.2.2}.
\end{proof}

\subsubsection{Preparatory estimates} \label{sec:3.2.1}
In this part, we present a series of estimates to be employed in the proof of Theorem \ref{global2}. First, we show that the support of the solution can be controlled by the time evolution of the initial support under the free  transport flow in the following lemma.

\begin{lemma}\label{desupp}
Suppose that $f_0\in W^{1,\infty}(\mathbb{R}^{2d})$ is a nonnegative initial datum with $\mathrm{supp}(f_0)\subseteq B_{R_0}\times B_{R_0}=:Q_0$ for some $R_0>0$. And let $f\in W^{1,\infty}([0,T]\times\mathbb{R}^{2d})$ be a nonnegative solution to \eqref{C-0}. If we define
\begin{equation} \label{NN-1}
Q(t)=\left\{(x+vt,v)|~(x,v)\in Q_0\right\}, \quad \mbox{for $t\in[0,T]$},
\end{equation}
then we have
\[ \mathrm{supp}(f(t))\subseteq Q(t) \quad \mbox{for any $t\in[0,T]$}. \]
\end{lemma}
\begin{proof}
We define
\[  R(t)\triangleq\sup\left\{|v|\big|v\in v\textup{-}\mathrm{supp}(f(t))\right\}, \quad S(t)\triangleq\sup\left\{|x-vt|\big|(x,v)\in\mathrm{supp}(f(t))\right\}. \]
Then it suffices to prove that $R(t)$ and $S(t)$ are both non-increasing functions with respect to $t$. Actually the monotonicity of $R(t)$ follows directly from Lemma ~\ref{desupp1}, but here we give a different proof provided that the initial datum is compactly supported.\newline

\noindent $\bullet$ Step D.1 (Almost everywhere differentiability of $R$ and $S$):~Since $f$ is a $W^{1,\infty}$ function with compact $v$-support, the vector field $\Psi_f$ is bounded in $[0,T]\times\mathbb{R}^d\times B_R$, which means the characteristic curve is spreading in a bounded velocity. It follows that $S(t)$ is also finite for $t\in[0,T]$, and that both $R(t)$ and $S(t)$ are Lipschitz continuous, thus differentiable almost everywhere.

\vspace{0.2cm}

\noindent $\bullet$ Step D.2 ($R(t)$ is non-increasing):~We fix $t_0\in(0,T)$ such that $R(t)$ is differentiable at $t_0$. Since $f_0$ has compact support, there exists a characteristic curve $Z_f(t,t_0;z_0)(z_0=(x_0,v_0))$ such that 
\[ |V_f(t_0,t_0;z_0)|=|v_0|=R(t_0). \]
By the definition of $R(t)$, 
\[ |V_f(t,t_0;z_0)|\leq R(t) \quad \mbox{for any $t\in(0,T)$}. \] 
Hence, we have
$$\frac{R^2(t)-R^2(t_0)}{t-t_0}\geq\frac{|V_f(t,t_0;z_0)|^2-|V_f(t_0,t_0;z_0)|^2}{t-t_0},\quad\forall~ t\in(t_0,T),$$
$$\frac{R^2(t)-R^2(t_0)}{t-t_0}\leq\frac{|V_f(t,t_0;z_0)|^2-|V_f(t_0,t_0;z_0)|^2}{t-t_0},\quad\forall~ t\in(0,t_0).$$
Letting $t\rightarrow t_0$ in two inequalities above, we get
\begin{align*}
    2R(t_0)R'(t_0)&=\frac{\mathrm{d}}{\mathrm{d}t}\Big|_{t=t_0}R^2(t) =\frac{\mathrm{d}}{\mathrm{d}t}\Big|_{t=t_0}|V_f(t,t_0;z_0)|^2 =-2\gamma v_0\cdot E[f](t_0,z_0)\\
    &=-2\gamma\int_{\mathbb{R}^d}v_0\cdot(v_0-v_*)f(t_0,x_0,v_*)\mathrm{d}v_* \\
    &\leq-2\gamma\int_{\mathbb{R}^d}(|v_0|^2-|v_0||v_*|)f(t_0,x_0,v_*)\mathrm{d}v_* \leq 0.
\end{align*}
Here we used the fact that $Z_f(t,t_0;z_0)$ is the solution to $\eqref{cheq}$ and that $|v_*|\leq|v_0|=R(t_0)$ for any $v_*$ in the $v$-support of $f(t_0)$. Thus we have
\[ R'(t_0)\leq 0, \]
which implies that $R(t)$ is non-increasing.

\vspace{0.2cm}

\noindent $\bullet$ Step D.3 ($S(t)$ is non-increasing):~Similarly, if we fix a differentiable point $t_0\in(0,T)$ of $S(t)$ and suppose that $|x-vt_0|$ reaches maximum at $z_0=(x_0,v_0)\in\mathrm{supp}(f(t_0))$, then we have
\begin{align*}
    2S(t_0)S'(t_0)&=\frac{\mathrm{d}}{\mathrm{d}t}\Big|_{t=t_0}S^2(t) =\frac{\mathrm{d}}{\mathrm{d}t}\Big|_{t=t_0}|X_f(t,t_0;z_0)-tV_f(t,t_0;z_0)|^2\\
    &=2(x_0-t_0v_0)\cdot(v_0-v_0+\gamma t_0E[f](t_0,z_0))\\
    &=2\gamma t_0\int_{\mathbb{R}^d}(x_0-t_0v_0)\cdot(v_0-v_*)f(t_0,x_0,v_*)\mathrm{d}v_* \\
    &=-2\gamma\int_{\mathbb{R}^d}(x_0-t_0v_0)\cdot\left[(x_0-t_0v_0)-(x_0-t_0v_*)\right]f(t_0,x_0,v_*)\mathrm{d}v_*\leq 0.
\end{align*}
This yields that $S(t)$ is also non-increasing.
\end{proof}
In the next lemma, we present coupled Gronwall-type differential inequalities to be useful in the proof of a global well-posedness.
\begin{lemma}\label{Gronwall}
Suppose that $a, b:[0,+\infty)\rightarrow [0,+\infty)$ be two measurable functions with
    $$\varepsilon:=\int_0^{+\infty} a(s)\mathrm{d}s<+\infty, \qquad\delta:=\int_0^{+\infty}(1+s) b(s)\mathrm{d}s<+\infty, $$
and let $A, B:[0,+\infty)\rightarrow [0,+\infty)$ be two absolutely continuous functions satisfying
    \begin{equation} 
    \begin{cases} \label{A}
    \displaystyle A'(t)\leq a(t) A(t)+b(t)A(t)B(t), \quad \mbox{a.e.}~t \in [0, \infty), \\
    \displaystyle  B'(t)\leq A(t)+a(t)B(t),
    \end{cases}
    \end{equation}
Then, there exists a positive constant $C = C(A(0), B(0))$ such that 
\[ A(t)\leq C, \quad  B(t)\leq C(1+t), \]
 provided that $\varepsilon$ and $\delta$ are small enough.
\end{lemma}

\begin{proof} 
We define
\[ \alpha(t)=\int_0^t a(s)\mathrm{d}s,\qquad\beta(t)=\int_0^t b(s)B(s)\mathrm{d}s, \quad t  \in[0,+\infty). \]
Then, it follows from \eqref{A} that 
\begin{equation}\label{A2}
    A(t)\leq A(0)\mathrm{e}^{\alpha(t)+\beta(t)}.
\end{equation}
Next we define 
\[ \displaystyle U(t)=B(t)\mathrm{e}^{-\alpha(t)-\beta(t)}, \]
and we use \eqref{A2} to find 
\begin{align*}
    U'(t)&=\mathrm{e}^{-\alpha(t)-\beta(t)}(B'(t)-\alpha'(t)B(t)-\beta'(t)B(t))\\
    &=\mathrm{e}^{-\alpha(t)-\beta(t)}(B'(t)-a(t)B(t)-b(t)B^2(t))\\
    &\leq\mathrm{e}^{-\alpha(t)-\beta(t)}(A(t)-b(t)B^2(t))\\
    &\leq\mathrm{e}^{-\alpha(t)-\beta(t)}A(t)\leq A(0).
\end{align*}
This yields
\[
U(t)\leq U(0)+A(0)t,
\]
which implies 
\begin{equation}\label{B2}
    B(t)\leq(B(0)+A(0)t)\mathrm{e}^{\alpha(t)+\beta(t)}\leq(B(0)+A(0)t)\mathrm{e}^{\varepsilon+\beta(t)}.
\end{equation}
We define $\displaystyle V(t)=\mathrm{e}^{\beta(t)}$. Then, it follows from \eqref{B2} that 
$$V'(t)=\beta'(t)V(t)=b(t)B(t)V(t)\leq\mathrm{e}^{\varepsilon}(B(0)+A(0)t) b(t) V^2(t).$$
Now we use $V(0)=1$ and the above differential inequality to find 
\[ V(t)\leq\left[1-\mathrm{e}^{\varepsilon}\int_0^t(B(0)+A(0)s)b(s) \mathrm{d}s\right]^{-1}\leq\left(1-C\delta\mathrm{e}^{\varepsilon}\right)^{-1}. \]
If $\varepsilon$ and $\delta$ are small enough, $V(t)$ and $\beta(t)$ are bounded. We insert it back into \eqref{A2} and $\eqref{B2}$ to get the desired estimates.
\end{proof}
Now we study a priori estimates for a global solution to $\eqref{C-0}$, which plays the most important role in the proof of Theorem \ref{global2}.
\begin{lemma}\label{global}
Suppose that $d\geq 2$, and $f_0\in W^{1,\infty}(\mathbb{R}^{2d})$ is a nonnegative initial datum with $\mathrm{supp}(f_0)\subseteq B_R\times B_R=:Q_0$, and let $f\in W^{1,\infty}([0,T]\times\mathbb{R}^{2d})$ be a nonnegative solution to \eqref{C-0}. Then there exists a positive constant $\gamma_0$ depending on $d, R$ and $\Vert f_0\Vert_{W^{1,\infty}}$ such that if $\gamma<\gamma_0$, we have
 \begin{equation} \label{NN-1-1}
        \Vert f(t)\Vert_{L^{\infty}}\leq C, \quad \Vert\nabla_xf(t)\Vert_{L^{\infty}}\leq C, \quad \Vert\nabla_vf(t)\Vert_{L^{\infty}}\leq C(1+t),
 \end{equation}
    where $C$ is also a constant depending on $d, R$ and $\Vert f_0\Vert_{W^{1,\infty}}$.
\end{lemma}
\begin{proof}  In the sequel, we provide the derivations of \eqref{NN-1-1} one by one. \newline

\noindent $\bullet$~Case A (Derivation of the first estimate \eqref{NN-1-1}):
Recall the time evolution of $Q_0$ under the free transport flow in \eqref{NN-1} by
\[ Q(t)=\left \{(x+vt,v)|~(x,v)\in Q_0\right \}. \]
By Lemma ~\ref{desupp}, we have 
\[ \mathrm{supp}(f(t))\subseteq Q(t) \quad \mbox{for any $t\in[0,T]$}. \]
For $x\in\mathbb{R}^d$ and $t\in[0,T]$, we define 
\[ {\mathcal S}(t,x)=\{v \mid (x,v)\in Q(t)\}. \]
Then, one has  
\[ v\in {\mathcal S}(t,x) \quad \iff \quad  v\in B_R \quad \mbox{and} \quad x-vt\in B_R. \]
This implies
\begin{equation} \label{NN-2}
\mathrm{diam} {\mathcal S}(t,x)\leq h(t):= \min\left\{2R,\frac{2R}{t}\right\}.
\end{equation}
For simplicity, we set
$$E(t,x,v):=E[f](t,x,v)=\int_{\mathbb{R}^d}(v-v_*)f(t,x,v_*)\mathrm{d}v_*. $$
For $(x,v)\in Q(t)$, we have
\begin{align*}
\begin{aligned}
|E(x,v,t)| &\leq\Vert f(t)\Vert_{L^{\infty}}\int_{{\mathcal S}(t,x)}|v-v_*|\mathrm{d}v_* \leq\Vert f(t)\Vert_{L^{\infty}}\int_{B_{h(t)}}|w|\mathrm{d}w \\
&=C\Vert f(t)\Vert_{L^{\infty}}h^{d+1}(t).
\end{aligned}
\end{align*}
Similarly, we also have
\begin{align*}
\begin{aligned}
& |\partial_{x_i}E(x,v,t)|\leq C\Vert\nabla_xf(t)\Vert_{L^{\infty}}h^{d+1}(t), \quad |\partial_{v_i}E(x,v,t)|\leq C\Vert f(t)\Vert_{L^{\infty}}h^d(t), \\
& \partial_{v_i,v_j}^2E(x,v,t)=0,\qquad|\partial_{x_i,v_j}^2E(x,v,t)|\leq C\Vert\nabla_xf(t)\Vert_{L^{\infty}}h^d(t), \quad i, j \in [d].
\end{aligned}
\end{align*}
Note that our equation \eqref{C-0} can be rewritten as
\begin{equation}\label{feq}
    \partial_tf+v\cdot\nabla_xf-\gamma E\cdot\nabla_vf=\gamma f\nabla_v\cdot E.
\end{equation}
Thus along the characteristic curve $Z_f(s,t;z)$, we have
$$\frac{\mathrm{d}}{\mathrm{d}s}f(s,Z_f(s,t;z))=\gamma(f\nabla_v\cdot E)(s,Z_f(s,t;z)).$$
If we define
$$F(t):=\Vert f(t)\Vert_{L^{\infty}},\quad t\in[0,T],$$
then using the estimates before, we have
$$F'(t)\leq\gamma\Vert f(t)\Vert_{L^{\infty}}\Vert\nabla_v\cdot E\Vert_{L^{\infty}}\leq C\gamma h^d(t)F^2(t).$$
This implies
$$F(t)\leq\Vert f_0\Vert_{L^{\infty}}\left(1-C\gamma\Vert f_0\Vert_{L^{\infty}}\int_0^t h^d(s)\mathrm{d}s\right)^{-1}.$$
The expression \eqref{NN-2} and the assumption $d\geq 2$ imply
\[  \int_0^{+\infty}h^d(s)\mathrm{d}s<+\infty. \]
Hence for $\gamma$ small enough, we get the first estimate in \eqref{NN-1-1}. \newline

\noindent $\bullet$~Case B  (Derivation of the second and third estimates \eqref{NN-1-1}):~We differentiate \eqref{feq} with respect to $x_i$ and $v_i$ respectively, to get
\begin{align}
\begin{aligned} \label{dxfeq}
& \partial_t\partial_{x_i}f+v\cdot\nabla_x\partial_{x_i}f-\gamma E\cdot\nabla_v\partial_{x_i}f=\gamma\partial_{x_i}f\nabla_v\cdot E+\gamma f\nabla_v\cdot\partial_{x_i}E+\gamma\partial_{x_i}E\cdot\nabla_vf, \\
& \partial_t\partial_{v_i}f+v\cdot\nabla_x\partial_{v_i}f-\gamma E\cdot\nabla_v\partial_{v_i}f=-\partial_{x_i}f+\gamma\partial_{v_i}E\cdot\nabla_vf+\gamma\partial_{v_i}f\nabla_v\cdot E.
\end{aligned}
\end{align}
We define
$$A(t):=\Vert\nabla_xf(t)\Vert_{L^{\infty}},\quad B(t):=\Vert\nabla_vf(t)\Vert_{L^{\infty}},\quad t\in[0,T].$$
Similarly, we use the previous estimates before to derive 
\begin{align*}
\begin{aligned}
& A'(t)\leq C\gamma h^d(t)A(t)+C\gamma h^{d+1}(t)A(t)B(t), \\
& B'(t)\leq A(t)+C\gamma h^d(t)B(t).
\end{aligned}
\end{align*}
We set
\[  {\mathfrak a}(t)=C\gamma h^d(t), \quad \mathfrak{b}(t)=C\gamma h^{d+1}(t) \]
Then, Lemma ~\ref{Gronwall} yields the desired second and third estimates in \eqref{NN-1-1}.
\end{proof} 
Now we are ready to provide  the proof of Theorem ~\ref{global2}. \newline

\subsubsection{Proof of Theorem \ref{global2}} \label{sec:3.2.2}
Let $\gamma_0$ be a positive constant in Lemma \ref{global}. Then for $\gamma<\gamma_0$, we define
$$T_0:=\sup\left\{T>0|~\exists~f\in W^{1,\infty}([0,T]\times\mathbb{R}^{2d})\textup{ with initial datum }~f_0~\textup{such that}~\eqref{supp}  \text{holds} \right\}.$$
Assume that $T_0<+\infty$. Then for all solutions $f\in W^{1,\infty}([0,T]\times\mathbb{R}^{2d})$ with initial data $f_0$ and satisfying \eqref{supp}, by Lemmas ~\ref{desupp} and ~\ref{global}, we have
$$v\textup{-}\mathrm{supp}(f(T))\subseteq B_{R_0},\qquad\Vert f(T)\Vert_{W^{1,\infty}}\leq C(1+T_0).$$
By the local well-posedness result (see Theorem~\ref{W1infty}), there exists a uniform time $t_0$ which does not depend on $f$, such that we can extend any solution above to time $T+t_0$. This fact contradicts to the definition of $T_0$. Thus we must have $T_0=+\infty$, which proves the existence part. The uniqueness part follows from Proposition ~\ref{uniqueness}.
\hfill $\square$
\subsection{Asymptotic completeness} \label{sec:3.3}
In this subsection, we study how the nonlinear equation \eqref{C-0} can be approximated by the corresponding free transport in large time.
This phenomena is called \textbf{asymptotic completeness} and it is an extensively studied topic in a scattering theory. \newline
Consider a nonlinear system 
\begin{equation}\label{nonlinear}
\partial_t f +v\cdot\nabla_x f=\mathcal{N}(f,\nabla_vf),
\end{equation}
and its associated free transport equation
\begin{equation}\label{ft}
\partial_t f +v\cdot\nabla_x f=0,
\end{equation}
we want to know whether or not the global solutions to \eqref{nonlinear} can be approximated by the free transport solutions to \eqref{ft} with suitable initial datum time-asymptotically. The free transport equation \eqref{ft} is solved by
$$f(t,x,v)=(U_0(t)f_0)(x,v):=f_0(x-vt,v),$$
where $f_0$ is the initial data and $U_0(t)$ is called the free transport operator. Now we give a rigorous definition for asymptotic completeness.
\begin{definition}\textup{
For $p \geq 1$, the nonlinear equation \eqref{nonlinear} exhibits $L^p$-asymptotic completeness if and only if for every global solution $f$ to \eqref{nonlinear}, there exist unique asymptotic free states $f_{\pm}$, such that
 \[ \lim_{t\rightarrow\pm\infty}\Vert U_0(-t)f(t)-f_{\pm}\Vert_{L^p}=0, \quad \mbox{or equivalently} \quad \lim_{t \to \pm \infty} \Vert f(t)-U_0(t)f_{\pm}\Vert_{L^p}=0. \]}
\end{definition}
The next proposition gives a sufficient condition for time-forward $L^1$-asymptotic completeness.
\begin{proposition}\label{sua}
\cite{CH11,HKLN07}
    Let $f$ be a global solution to \eqref{nonlinear} satisfying the estimate:
    \begin{equation}\label{su}
        \int_0^{+\infty}\Vert U_0(-t)\mathcal{N}(f,\nabla_vf)(t)\Vert_{L^1}\mathrm{d}t<+\infty.
    \end{equation}
Then, there exists a unique $L^1$-scattering state $f_0^+\in L^1(\mathbb{R}^{2d})$ such that
    $$\lim_{t\rightarrow +\infty}\Vert U_0(-t)f(t)-f_0^+\Vert_{L^1}=0.$$
\end{proposition}

\begin{proof} 
Using Duhamel's formula for \eqref{nonlinear}, we get
$$f(t)=U_0(t)f_0+\int_0^t U_0(t-\tau)\mathcal{N}(f,\nabla_v f)(\tau)\mathrm{d}\tau.$$
Thus, we have
$$U_0(-t)f(t)=f_0+\int_0^t U_0(-\tau)\mathcal{N}(f,\nabla_v f)(\tau)\mathrm{d}\tau.$$
Define
$$f_+^0:= f_0+\int_0^{+\infty} U_0(-\tau)\mathcal{N}(f,\nabla_v f)(\tau)\mathrm{d}\tau,$$
which is well-defined by \eqref{su}. Then, we have
$$\Vert U_0(-t)f(t)-f_+^0\Vert_{L^1}
\leq\int_t^{+\infty}\Vert U_0(-\tau)\mathcal{N}(f,\nabla_v f)(\tau)\Vert_{L^1}\mathrm{d}\tau\rightarrow 0,$$
as $t\rightarrow +\infty$. The uniqueness is obvious, since $f_0^+$ is the $L^1$-limit of $U_0(-t)f(t)$ as $t\rightarrow +\infty$.
\end{proof} 
Using the estimates obtained in, for instance, Lemma~\ref{global}, Proposition~\ref{sua} admits the following direct corollary: the forward-in-time $L^1$ asymptotic completeness for our equation~\eqref{Eq}.
\begin{theorem}

Suppose that $d\geq 2$, and $f_0\in W^{1,\infty}(\mathbb{R}^{2d})$ is a nonnegative initial datum with compact support,  and let $\gamma>0$ be small enough such that the global solution $f$ in Theorem \textup{~\ref{global2}} exists. Then there exists a unique $L^1$-scattering state $f_0^+\in L^1(\mathbb{R}^{2d})$, such that
    $$\lim_{t\rightarrow +\infty}\Vert U_0(-t)f(t)-f_0^+\Vert_{L^1}=0.$$
\end{theorem}
{\raggedright\textit{Proof.}\quad}
By Proposition~\ref{sua}, it suffices to show that
$$\int_0^{+\infty}\Vert\gamma\nabla_v\cdot(E[f]f)(t)\Vert_{L^1}\mathrm{d}t<+\infty, $$
Suppose that 
\[ \mathrm{supp}(f_0)\subseteq B_R\times B_R, \]
and we recall that
$$\mathrm{supp}(f(t))\subseteq Q(t)=\{(x+vt,v)|(x,v)\in B_R\times B_R\}.$$
Since the Lebesgue measure of $Q(t)$ does not change along with $t$, we have
$$\Vert\nabla_v\cdot(E[f]f)(t)\Vert_{L^1}\leq C\Vert\nabla_v\cdot(E[f]f)(t)\Vert_{L^{\infty}},$$
where $C$ is some constant depending on $d$ and $R$. Moreover, we  recall that
\begin{align*}
\begin{aligned}
& \Vert f(t)\Vert_{L^{\infty}}\leq C,\quad\Vert\nabla_vf(t)\Vert_{L^{\infty}}\leq C(1+t), \\
& \Vert E(t)\Vert_{L^{\infty}}\leq C\Vert f(t)\Vert_{L^{\infty}}h^{d+1}(t),\quad\Vert\nabla_v\cdot E(t)\Vert_{L^{\infty}}\leq C\Vert f(t)\Vert_{L^{\infty}}h^d(t),
\end{aligned}
\end{align*}
where $h$ is defined in \eqref{NN-2}. Therefore, we have
\begin{align*}
\begin{aligned}
\Vert\nabla_v\cdot(E[f]f)(t)\Vert_{L^{\infty}} &\leq\Vert\nabla_v\cdot E(t)\Vert_{L^{\infty}}\Vert f(t)\Vert_{L^{\infty}}+\Vert E(t)\Vert_{L^{\infty}}\Vert\nabla_vf(t)\Vert_{L^{\infty}} \\
&\leq Ch^d(t)+C(1+t)h^{d+1}(t).
\end{aligned}
\end{align*}
The right hand side is integrable, which yields the desired estimate. 
\hfill $\square$

\subsection{Mono-kinetic solutions} \label{sec:3.4}
In this subsection, we study mono-kinetic solutions to \eqref{Eq}, which refer to weak solutions of the form:
\begin{equation}\label{form}
    f(t,x,v)=\rho(t,x) \otimes \delta_{v=u(t,x)}.
\end{equation}
Next, we show that in the one-dimensional case, $u(t,x)$ solves the inviscid Burgers' equation, which coincides with what we have derived in the introduction. 
Furthermore, we can see that mono-kinetic solutions can blow up in finite time, thanks to the property of inviscid Burgers' equation. 

Note that $f$ is a weak solution to \eqref{Eq} if and only if for any $g\in C_0^{\infty}(\mathbb{R}^{2d})$,
\begin{align}
\begin{aligned} \label{weak}
& \frac{d}{dt} \int_{\mathbb{R}^{2d}}f(t,x,v)g(x,v)\mathrm{d}x\mathrm{d}v \\
& \hspace{1cm} = \int_{\mathbb{R}^{2d}}(v\cdot\nabla_xg)f\mathrm{d}x\mathrm{d}v -\gamma\int_{\mathbb{R}^{2d}}f\nabla_vg\cdot\int_{\mathbb{R}^d}(v-v_*)f(t,x,v_*)\mathrm{d}v_* \mathrm{d}x\mathrm{d}v.
\end{aligned}
\end{align}
We substitute \eqref{form} into \eqref{weak}. 
Note that the second term of right-hand side of \eqref{weak} becomes $0$, so
\begin{equation}\label{weak2}
    \frac{d}{dt} \int_{\mathbb{R}^d}\rho(t,x)g(x,u(t,x))\mathrm{d}x =  \int_{\mathbb{R}^d}\rho(t,x)u(t,x)\cdot(\nabla_xg)(x,u(t,x))\mathrm{d}x.
\end{equation}

\noindent $\bullet$ (Derivation of the continuity equation): We choose $g$ to be  independent of $v$ to see that $\rho$ and $u$ solve the continuity equation:
\begin{equation}\label{continuity}
    \partial_t\rho+\nabla_x\cdot(\rho u)=0.
\end{equation}

\vspace{0.2cm}

\noindent $\bullet$ (Derivation of the momentum equation):~We substitute \eqref{continuity} into \eqref{weak2} to find 
\begin{align*}
    0=&\int_{\mathbb{R}^d}\partial_t\rho\cdot g(x,u(t,x))\mathrm{d}x+\int_{\mathbb{R}^d}\rho\partial_tu\cdot(\nabla_vg)(x,u(t,x))\mathrm{d}x-\int_{\mathbb{R}^d}\rho u\cdot(\nabla_xg)(x,u(t,x))\mathrm{d}x\\
    =&\int_{\mathbb{R}^d}\rho\partial_tu\cdot(\nabla_vg)(x,u(t,x))\mathrm{d}x+\int_{\mathbb{R}^d}\rho u\cdot\left[\nabla_xg(x,u(t,x))-(\nabla_xg)(x,u(t,x))\right]\mathrm{d}x\\
    =&\int_{\mathbb{R}^d}\rho\partial_tu\cdot(\nabla_vg)(x,u(t,x))\mathrm{d}x+\int_{\mathbb{R}^d}\rho u\cdot\left(D_xu\cdot(\nabla_vg)(x,u(t,x))\right)\mathrm{d}x,
\end{align*}
where $D_xu$ denotes the Jacobian matrix of $u$ with respect to $x$. Now, we take $\nabla_vg$ to be independent of $v$ to see that $u$ solves the equation:
\begin{equation}\label{dimd}
    \partial_tu+(D_xu)^Tu=0.
\end{equation}
For the one-dimensional case,  the equation \eqref{dimd} becomes
\begin{equation}
    \partial_tu+u\partial_xu=0.
\end{equation}
This is exactly the inviscid Burgers' equation, whose solutions are known to blow up in finite time for some cases. 
For a more concrete example, when the initial data is given by $u_0(x)=-x$, the solution is
$$u(t,x)=\frac{x}{t-1}.$$
We insert this explicit form back into \eqref{continuity} to  get
$$\rho(t,x)=\frac{1}{1-t}\rho_0\left(\frac{x}{1-t}\right),$$
where $\rho_0$ is the initial datum. So, we have
$$f(t,x,v)=\frac{1}{1-t}\rho_0\left(\frac{x}{1-t}\right)\delta_{v=\frac{x}{t-1}},$$
which blows up at time $t=1$.

\section{Numerical simulations} \label{sec:4}
In this section, we present numerical simulations in the one-dimensional case with $d= 1$. The goal is to visualize the time evolution of several physical quantities and to confirm the role of the alignment coefficient $\gamma$.

\subsection{Computational setup} \label{sec:5.1}

We solve \eqref{Eq} on a truncated phase-space domain
\[
(x,v)\in [0,L_x]\times[-L_v/2,L_v/2],
\]
with $L_x,L_v>0$ chosen large enough so that the numerical support stays away from the boundary
on the time interval $[0,T]$ considered.
This truncation is consistent with the compact-support framework of the analysis:
if $f_0$ is compactly supported, then the support remains bounded on any finite time interval.
We use a uniform grid with mesh sizes $\Delta x$ and $\Delta v$ and time step $\Delta t$, and we write $f_{i,j}^n\approx f(t^n,x_i,v_j)$.
The initial datum $f_0$ is taken to be a uniform distribution on a square patch $S$ inside the domain.
We implement \eqref{Eq} by an operator splitting method, separating the transport and alignment parts.
The hyperparameters are shown in Table~\ref{tab:params}.
\begin{table}[h!]
\centering
\begin{tabular}{@{}ll@{}}
\toprule
\textbf{Quantity} & \textbf{Value} \\
\midrule
Truncated domain size ($L_x,\ L_v$) & $(20,6)$ \\
Final time $T$ & $3.0$ \\
Time step $\Delta t$ & $10^{-4}$ \\
Grid size $(\Delta x,\ \Delta v)$ & $(0.05, 0.01)$ \\
Length of side of $S$ (support of initial data) & $2$ \\
Position of center of $S$ & $(11,-0.3) $ \\
\bottomrule
\hspace{0.2cm}
\end{tabular}
\caption{Common simulation parameters.}
\label{tab:params}
\end{table}

\subsection{Physical observables and alignment} \label{sec:5.2}
In the sequel, we track the discrete analogues of the conserved/dissipated quantities in Proposition~\ref{evolution}. This is done for $\gamma=1$.
Figure~\ref{Fig:Physical_quantities} shows the evolution of mass, momentum, energy, and entropy.
We observe that, after some time, numerical errors accumulate and the computed quantities no longer perfectly follow the theoretical evolution laws; however, the expected behavior is clearly visible at early times. 

\begin{figure}[htbp]
    \centering
    \begin{minipage}{0.48\textwidth}
        \centering
        \includegraphics[width=\linewidth]{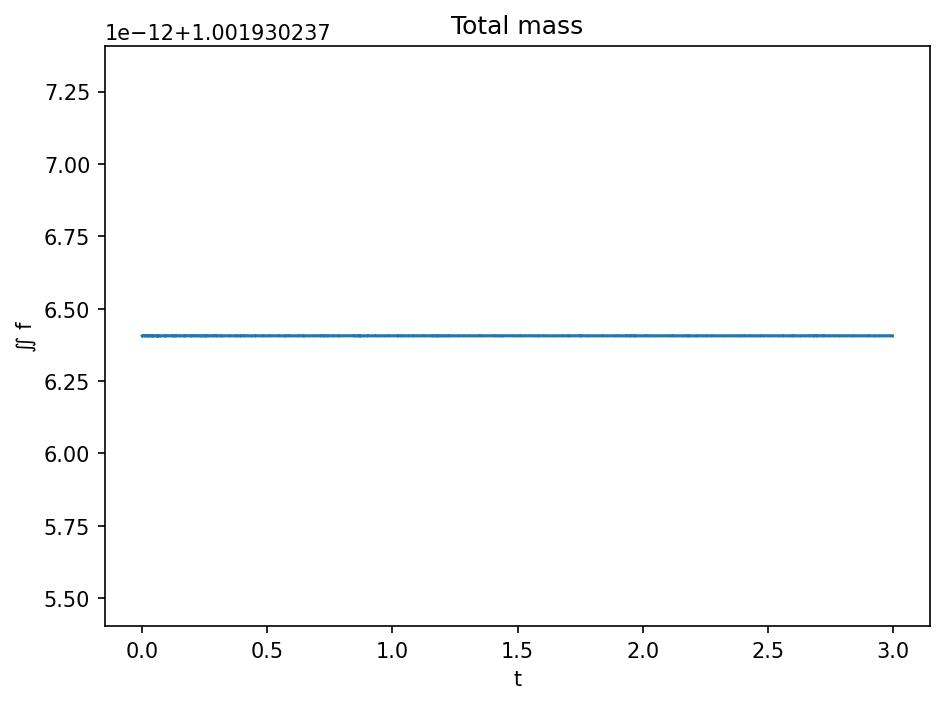}
    \end{minipage}\hfill
    \begin{minipage}{0.48\textwidth}
        \centering
        \includegraphics[width=\linewidth]{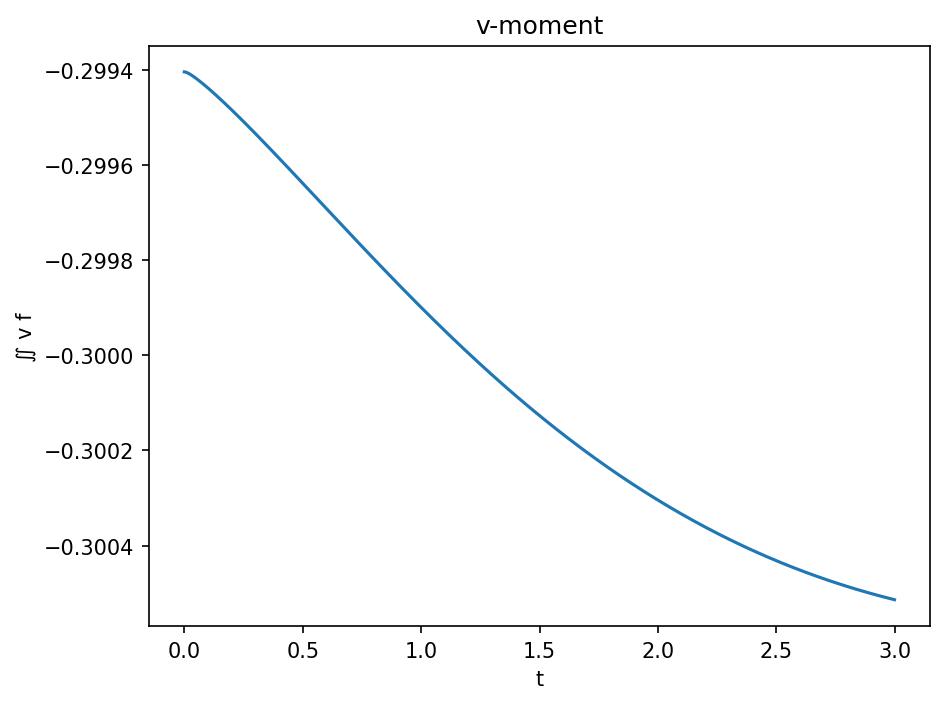}
    \end{minipage}

    \vspace{0.5em}

    \begin{minipage}{0.48\textwidth}
        \centering
        \includegraphics[width=\linewidth]{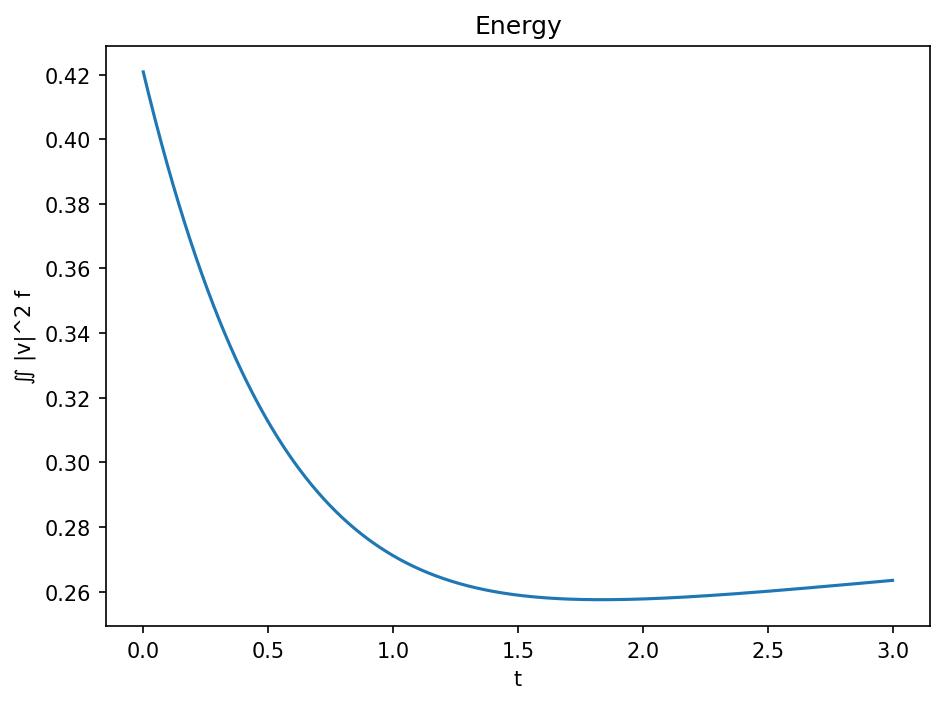}
    \end{minipage}\hfill
    \begin{minipage}{0.48\textwidth}
        \centering
        \includegraphics[width=\linewidth]{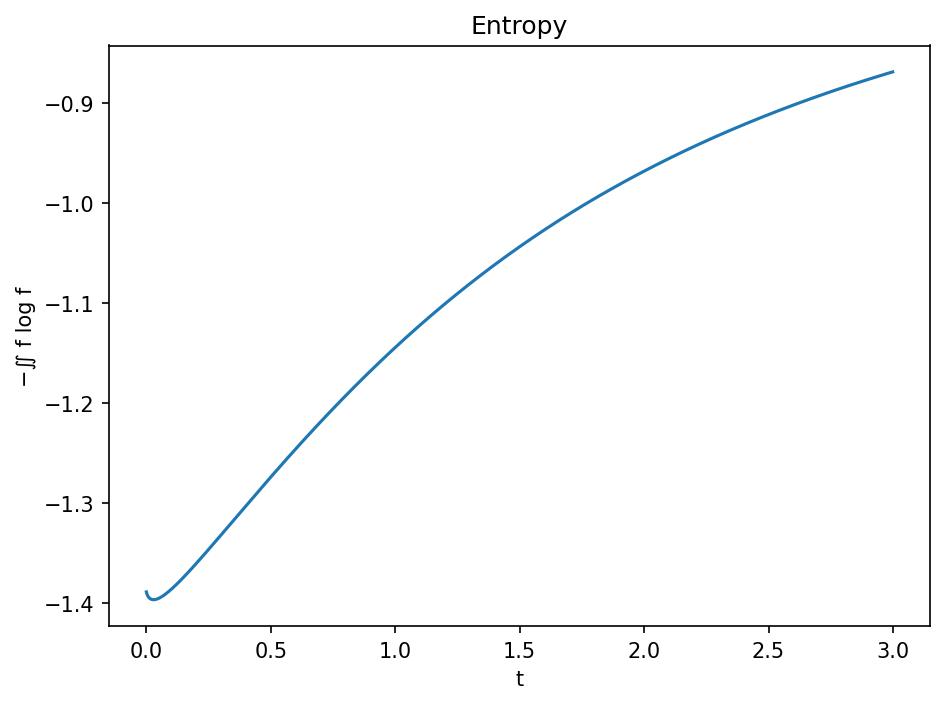}
    \end{minipage}
    \caption{From the left top, mass, velocity momentum, energy, and entropy.
    Although the velocity moment looks like it is decreasing, the difference is not that big compare to it's value.}
    \label{Fig:Physical_quantities}
\end{figure}
Next, we illustrate the effect of $\gamma$ by considering
\[
h(t,v)=\sup_{x\in [-L_x,L_x]} f(t,x,v).
\]
$v$-support of $h$ is actually the $v$-support of $f$, so we can visually see the $v$-support of $f$.
Also, the amplitude shows (Figure~\ref{Fig:h}) that  for small $\gamma$ the time evolution of $h(t,\cdot)$ is gradual, whereas for large $\gamma$ the change is much more rapid.

\begin{figure}[htbp]
    \centering
    \begin{minipage}{0.32\textwidth}
        \centering
        \includegraphics[width=\linewidth]{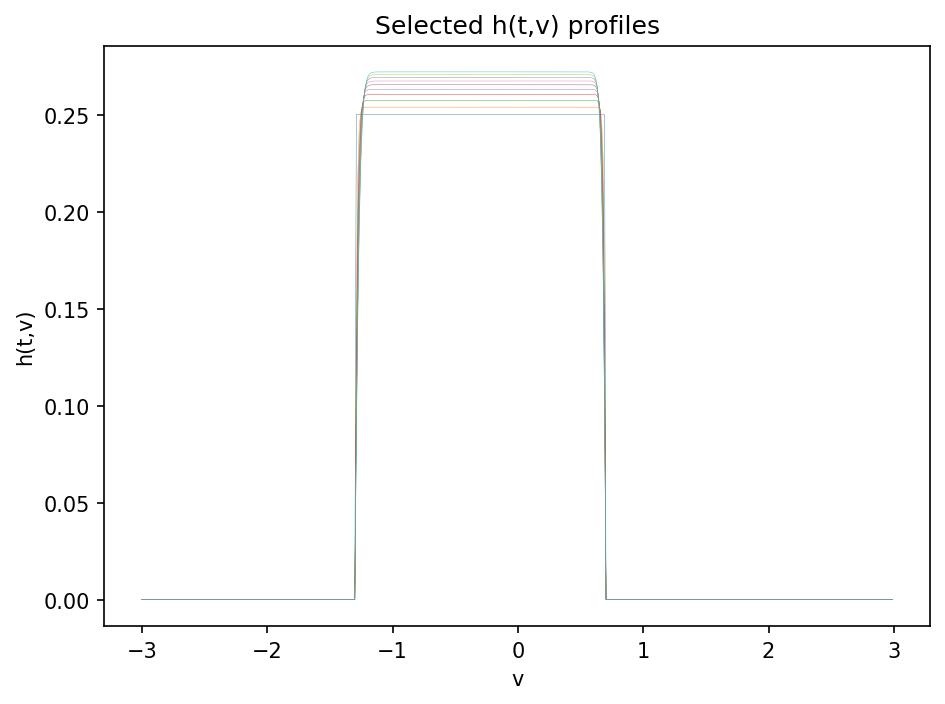}
    \end{minipage}\hfill
    \begin{minipage}{0.32\textwidth}
        \centering
        \includegraphics[width=\linewidth]{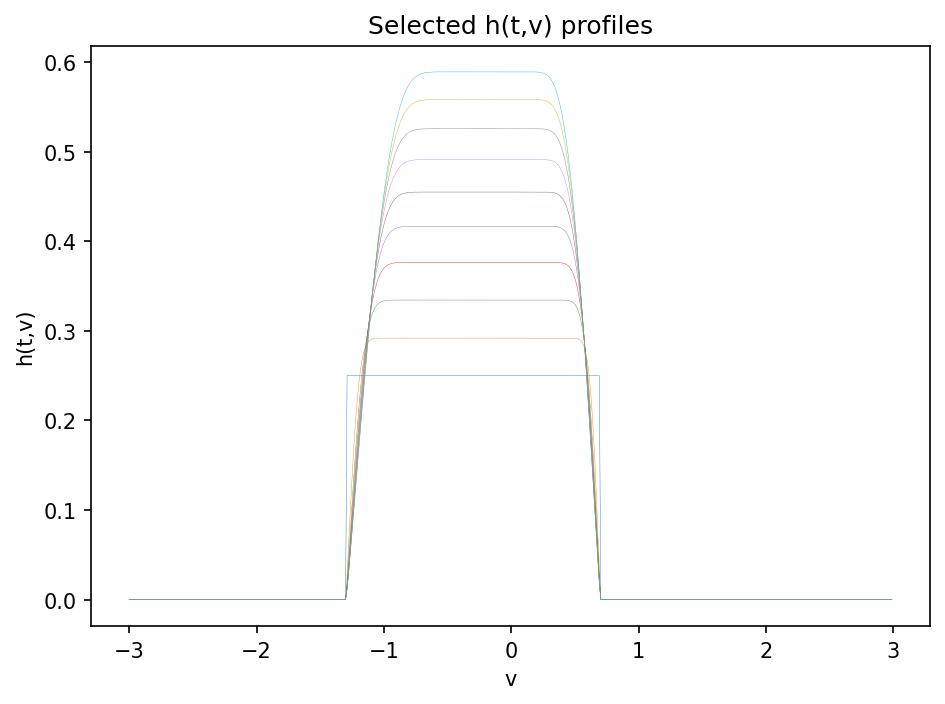}
    \end{minipage}
    \begin{minipage}{0.32\textwidth}
        \centering
        \includegraphics[width=\linewidth]{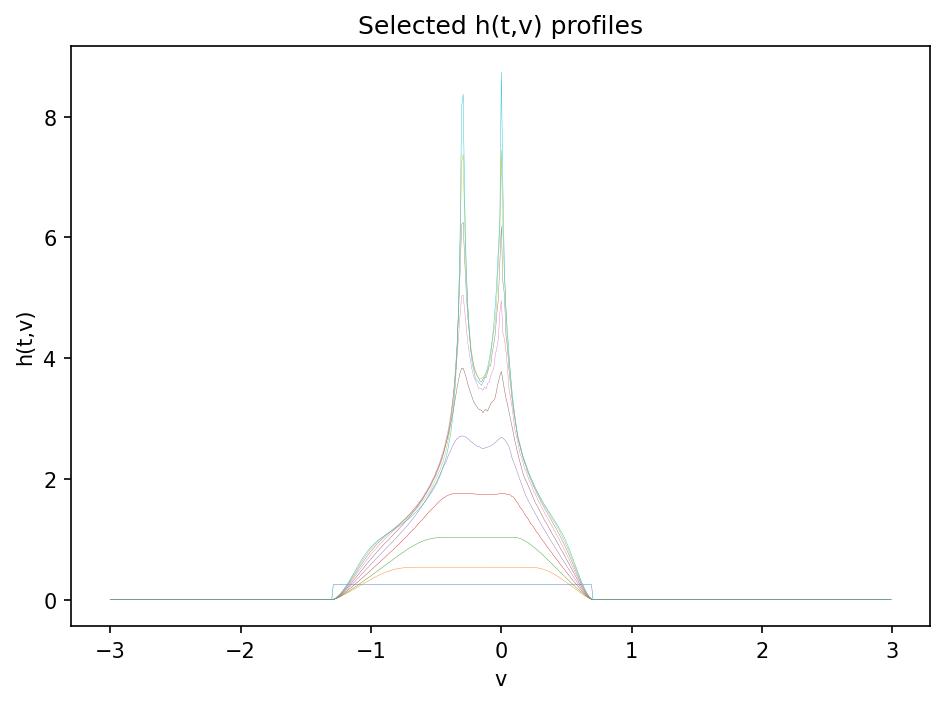}
    \end{minipage}
    \caption{Profile of the function $h(t,v)$ with $\gamma=0.1,\ 1.0,\text{ and } 5.0$. 
    As $t$ increases the profile's amplitude goes higher.
    This agrees with the alignment property of our equation (in momentum or energy), since it means that the distribution aligns in velocity.}
    \label{Fig:h}
\end{figure}
In particular, if we compare the final distributions via heatmaps (Figure~\ref{Fig:heatmap}), the small-$\gamma$ case resembles free transport, while the large-$\gamma$ case is closer to a mono-kinetic profile.
The former was more predictable from Proposition~\ref{sua}, but the  latter is more interesting.
This possibly motivates further investigations of the stability of mono-kinetic solutions in the large-$\gamma$ regime.

\begin{figure}[htbp]
    \centering
    \includegraphics[width=0.9\linewidth]{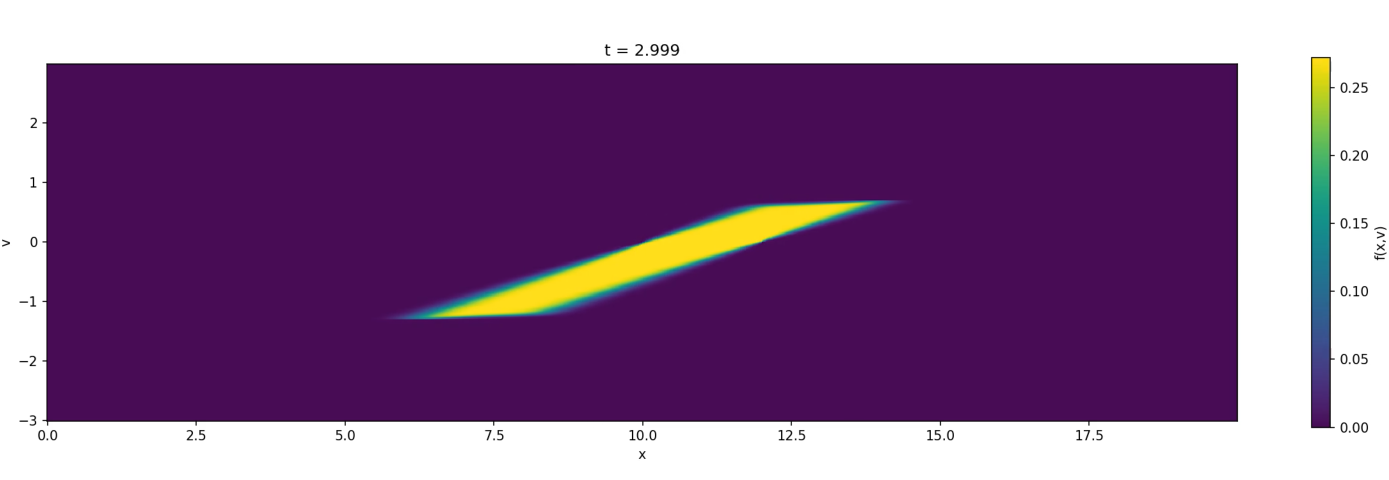}
    \vspace{0.5em}
    \includegraphics[width=0.9\linewidth]{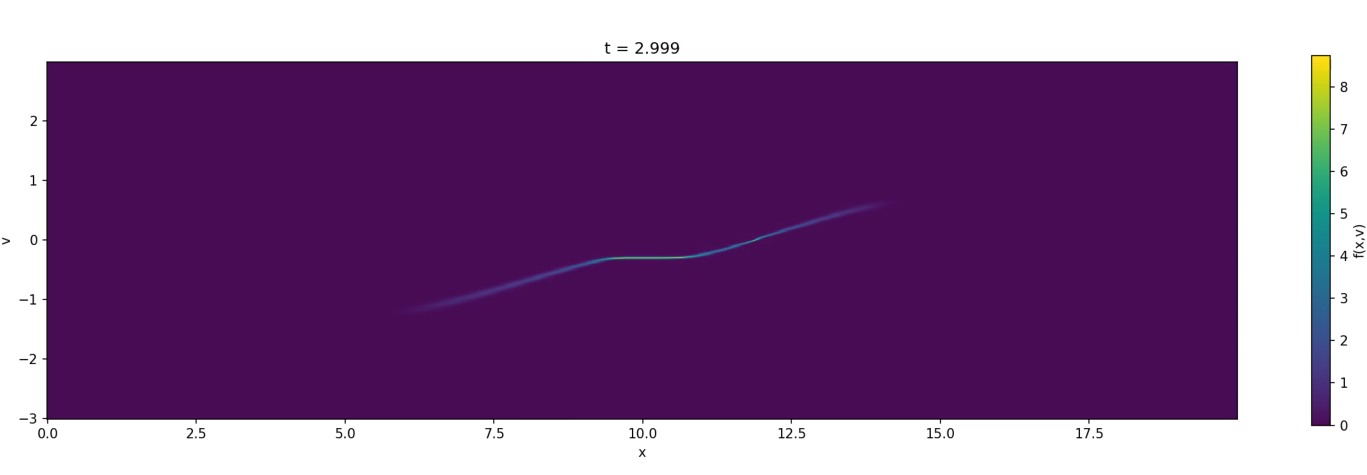}
    \caption{Each final distribution is produced by $\gamma=0.1,\ 5.0$, respectively.}
    \label{Fig:heatmap}
\end{figure}

\section{Conclusion} \label{sec:5}
In this paper, we have derived the kinetic equation arising from the large-scale limit from the Cucker-Smale model. For the derived kinetic model, we have obtained the local and global existence of solutions in different function spaces $W^{1,\infty}$ and $C_b^{1,\alpha}$. We also showed that the given nonlinear system can be effectively approximated by the corresponding free transport flow. Of course, there are several open questions that we did not discuss in this work. To name a few, we did not treat the emergent dynamics of the proposed kinetic equation, and the derivations of the kinetic and hydrodynamic models from the Cucker-Smale model are formal. These interesting issues will be left for future work.

\begin{appendix}

\section{Formal derivation of the kinetic model}\label{App-A}
\setcounter{equation}{0}
In this appendix, we derive~\eqref{Eq} in a very classical way by using the BBGKY hierarchy together with a molecular chaos assumption.
We set
$$X_N=(x_1,\cdots,x_N),\quad V_N=(v_1,\cdots,v_N),$$
and
$$X_j=(x_1,\cdots,x_j),\quad V_j=(v_1,\cdots,v_j),$$
for $j=1\cdots N$.  Let $f^N(t,X_N,V_N)$ be a symmetric $N$-particle phase-space density on $(\mathbb{R}^{2d})^N$
evolving according to the particle system~\eqref{RCS}. Then $f^N$ satisfies the Liouville equation
\begin{equation}\label{Liouville}
	\Big(\partial_t+\sum_{i=1}^{N}v_i\cdot\nabla_{x_i}\Big)f^N=\kappa T_N^{\varepsilon}f^N,
\end{equation}
where
\[ T_j^{\varepsilon}=\sum_{1\leq k<l\leq j}T_{k,l}^{\varepsilon}, \quad T_{k,l}^{\varepsilon}f^N=(\nabla_{v_k}-\nabla_{v_l})\cdot\left(f^N\psi\left(\frac{|x_k-x_l|}{\varepsilon}\right)(v_k-v_l)\right). \]
Now we define the $j$-marginal distribution $f_j^N$ as
$$f_j^N(t,X_j,V_j)=\int_{\mathbb{R}^{2d(N-j)}}f^N(X_j,x_{j+1},\cdots,x_N;V_j,v_{j+1},\cdots,v_N)\mathrm{d}x_{j+1}\cdots\mathrm{d}x_N\mathrm{d}v_{j+1}\cdots\mathrm{d}v_N$$
for $j \in [N-1]$ and $f_N^N=f^N$. Then by a partial integration of the Liouville equation \eqref{Liouville} we get the BBGKY hierarchy:
\begin{equation}\label{BBGKY}
	\left(\partial_t+\sum_{i=1}^{j}v_i\cdot\nabla_{x_i}\right)f_j^N=\kappa T_j^{\varepsilon}f_j^N+\kappa(N-j)C_{j+1}^{\varepsilon}f_{j+1}^N,\quad j \in [N],
\end{equation}
where
$$C_{j+1}^{\varepsilon}=\sum_{i=1}^{j}C_{i,j+1}^{\varepsilon},
$$
with
$$(C_{i,j+1}^{\varepsilon}f_{j+1}^N)(X_j,V_j)=\int_{\mathbb{R}^{2d}}\nabla_{v_i}\cdot\left(f_{j+1}^N\psi\left(\frac{|x_i-x_{j+1}|}{\varepsilon}\right)(v_i-v_{j+1})\right)\mathrm{d}x_{j+1}\mathrm{d}v_{j+1}.$$
Taking $j=1$ in \eqref{BBGKY}, and since $T_1^{\varepsilon}=0$, we get
\begin{equation}\label{j1}
	\left(\partial_t+v_1\cdot\nabla_{x_1}\right)f_1^N=\kappa(N-1)C_{2}^{\varepsilon}f_{2}^N,
\end{equation}
where
$$(C_2^{\varepsilon}f_2^N)(x_1,v_1)=\int_{\mathbb{R}^{2d}}\nabla_{v_1}\cdot\left(f_2^N\psi\left(\frac{|x_1-x_2|}{\varepsilon}\right)(v_1-v_2)\right)\mathrm{d}x_2\mathrm{d}v_2.$$
After the change of variables $x_2=x_1+\varepsilon r$, we get
\begin{align}
\begin{aligned} \label{change}
& \kappa(N-1)(C_{2}^{\varepsilon}f_{2}^N)(x_1,v_1) \\
& \hspace{1cm} =\kappa(N-1)\varepsilon^d\int_{\mathbb{R}^{2d}}\nabla_{v_1}\cdot\left(\psi(|r|)(v_1-v_2)f_2^N(t,x_1,x_1+\varepsilon r,v_1,v_2)\right)\mathrm{d}r\mathrm{d}v_2.
\end{aligned}
\end{align}
Then if we assume that $f_1^N(t)\rightarrow f(t)$, $f_j^N(t)\rightarrow f(t)^{\otimes j}$ as $\varepsilon\rightarrow 0$ and $f_j^N$ has enough regularity, then by passing $\varepsilon\rightarrow 0$ in \eqref{change}, we see that the right hand side of \eqref{j1} will converge to $\displaystyle\gamma\nabla_v\cdot\left(E[f]f\right)(x_1,v_1)$, where
$$\gamma=\kappa\int_{\mathbb{R}^d}\psi(|r|)\mathrm{d}r.$$
Letting $\varepsilon\rightarrow 0$ in \eqref{j1}, we see that the equation solved by $f$ is exactly \eqref{Eq}.\par

\end{appendix}

\bibliographystyle{abbrv}
\bibliography{reference}

@article{HKLN07, 
title={Asymptotic completeness for relativistic kinetic equations with short-range interaction forces}, 
author={Ha, Seung-Yeal and Kim, Yong Duck and Lee, Ho and Noh, Se Eun}, 
journal={Methods and Applications of Analysis}, 
volume={14}, 
number={3}, 
pages={251–-262},
year={2007}
}

@article{CH11,
  title={Asymptotic behavior of the nonlinear {V}lasov equation with a self-consistent force},
  author={Choi, Sun-Ho and Ha, Seung-Yeal},
  journal={SIAM journal on mathematical analysis},
  volume={43},
  number={5},
  pages={2050--2077},
  year={2011},
}

@article{BCP97,
  title={A kinetic equation for granular media},
  author={Benedetto, Dario and Caglioti, Emanuele and Pulvirenti, Mario},
  journal={ESAIM: Mathematical Modelling and Numerical Analysis},
  volume={31},
  number={5},
  pages={615--641},
  year={1997},
}

@article{BCC+08,
  author  = {Ballerini, Michele and Cabibbo, Nicol{\`o} and Candelier, Raphael and Cavagna, Andrea and Cisbani, E. and Giardina, Irene and Lecomte, V. and Orlandi, A. and Parisi, Giorgio and Procaccini, Andrea and Viale, M. and Zdravkovi{\'c}, V. and others},
  title   = {Interaction ruling animal collective behavior depends on topological rather than metric distance: Evidence from a field study},
  journal = {Proceedings of the National Academy of Sciences},
  volume  = {105},
  number  = {4},
  pages   = {1232--1237},
  year    = {2008}
}

@article{TKI+13,
  author  = {Tunstr{\o}m, K. and Katz, Y. and Ioannou, C. C. and Huepe, C. and Lutz, M. J. and Couzin, I. D.},
  title   = {Collective states, multistability and transitional behavior in schooling fish},
  journal = {PLOS Computational Biology},
  volume  = {9},
  number  = {2},
  pages   = {e1002915},
  year    = {2013},
}

@article{Buc68,
author = {John Buck},
 journal = {The Quarterly Review of Biology},
 number = {3},
 pages = {265--289},
 publisher = {The University of Chicago Press},
 title = {Synchronous Rhythmic Flashing of Fireflies. II},
 volume = {63},
 year = {1988}
}

@InProceedings{Kur05,
  author    = {Kuramoto, Yoshiki},
  title     = {Self-entrainment of a population of coupled non-linear oscillators},
  booktitle = {International Symposium on Mathematical Problems in Theoretical Physics},
  series    = {Lecture Notes in Physics},
  volume    = {39},
  pages     = {420--422},
  publisher = {Springer},
  year      = {1975}
}

@book{Kur84,
  author    = {Kuramoto, Yoshiki},
  title     = {Chemical Oscillations, Waves, and Turbulence},
  series    = {Springer Series in Synergetics},
  publisher = {Springer},
  year      = {1984},
  pages = {158}
}

@article{VCB+95,
  author  = {Vicsek, Tam{\'a}s and Czir{\'o}k, Andr{\'a}s and Ben-Jacob, Eshel and Cohen, Inon and Shochet, Ofer},
  title   = {Novel Type of Phase Transition in a System of Self-Driven Particles},
  journal = {Physical Review Letters},
  volume  = {75},
  number  = {6},
  pages   = {1226--1229},
  year    = {1995}
}

@article{VZ12,
  author  = {Vicsek, Tam{\'a}s and Zafeiris, Anna},
  title   = {Collective motion},
  journal = {Physics Reports},
  volume  = {517},
  number  = {3},
  pages   = {71--140},
  year    = {2012}
}

@article{CS07,
  author  = {Cucker, Felipe and Smale, Steve},
  title   = {Emergent Behavior in Flocks},
  journal = {IEEE Transactions on Automatic Control},
  volume  = {52},
  number  = {5},
  pages   = {852--862},
  year    = {2007}
}

@article{RLX15,
author = {Ru, Lining and Li, Zhuchun and Xue, Xiaoping},
title = {Cucker-Smale Flocking with Randomly Failed Interactions},
journal = {Journal of the Franklin Institute},
volume = {352},
number = {3},
pages = {1099-1118},
year = {2015}
}

@article{HM19,
  author  = {Yuehua He and Xiaowu Mu},
  title   = {Cucker--{S}male flocking subject to random failure on general digraphs},
  journal = {Automatica},
  volume  = {106},
  pages   = {54--60},
  year    = {2019}
}

@article{MT11,
  author  = {Motsch, S{\'e}bastien and Tadmor, Eitan},
  title   = {A New Model for Self-organized Dynamics and Its Flocking Behavior},
  journal = {Journal of Statistical Physics},
  volume  = {144},
  number  = {5},
  pages   = {923--947},
  year    = {2011}
}

@article{AH10,
  author  = {Ahn, Shin Mi and Ha, Seung-Yeal},
  title   = {Stochastic flocking dynamics of the {C}ucker--{S}male model with multiplicative white noises},
  journal = {Journal of Mathematical Physics},
  volume  = {51},
  number  = {10},
  pages   = {17},
  year    = {2010}
}

@article{HT08,
  author  = {Ha, Seung-Yeal and Tadmor, Eitan},
  title   = {From particle to kinetic and hydrodynamic descriptions of flocking},
  journal = {Kinetic and Related Models},
  volume  = {1},
  number  = {3},
  pages   = {415--435},
  year    = {2008}
}

@article{HL09,
  author  = {Ha, Seung-Yeal and Liu, Jian-Guo},
  title   = {A simple proof of the {C}ucker--{S}male flocking dynamics and mean-field limit},
  journal = {Communications in Mathematical Sciences},
  volume  = {7},
  number  = {2},
  pages   = {297--325},
  year    = {2009}
}

@article{HKZ18,
  author  = {Ha, Seung-Yeal and Kim, Jeongho and Zhang, Xiongtao},
  title   = {Uniform stability of the {C}ucker--{S}male model and its application to the mean-field limit},
  journal = {Kinetic and Related Models},
  volume  = {11},
  number  = {5},
  pages   = {1157--1181},
  year    = {2018}
}

@incollection{JabinWangActiveParticles2017,
	author    = {Jabin, Pierre-Emmanuel and Wang, Zhenfu},
	title     = {Mean Field Limit for Stochastic Particle Systems},
	booktitle = {Active Particles, Volume 1: Advances in Theory, Models, and Applications},
	editor    = {Bellomo, Nicola and Degond, Pierre and Tadmor, Eitan},
	series    = {Modeling and Simulation in Science, Engineering and Technology},
	publisher = {Birkh{\"a}user Basel},
	year      = {2017},
	pages     = {379--402},
	doi       = {10.1007/978-3-319-49996-3_10}
}

@incollection{SznitmanSaintFlour1991,
	author    = {Sznitman, Alain-Sol},
	title     = {Topics in Propagation of Chaos},
	booktitle = {Ecole d'Et{\'e} de Probabilit{\'e}s de Saint-Flour XIX --- 1989},
	editor    = {Hennequin, Paul-Louis},
	series    = {Lecture Notes in Mathematics},
	volume    = {1464},
	publisher = {Springer},
	address   = {Berlin, Heidelberg},
	year      = {1991},
	pages     = {165--251},
	doi       = {10.1007/BFb0085169}
}

@article{SznitmanBurgers1986,
	author    = {Sznitman, Alain-Sol},
	title     = {A Propagation of Chaos Result for {B}urgers' Equation},
	journal   = {Probability Theory and Related Fields},
	volume    = {71},
	year      = {1986},
	pages     = {581--613},
	doi       = {10.1007/BF00699042}
}

@article{Oelschlager1985,
	author    = {Oelschl{\"a}ger, Klaus},
	title     = {A Law of Large Numbers for Moderately Interacting Diffusion Processes},
	journal   = {Zeitschrift f{\"u}r Wahrscheinlichkeitstheorie und Verwandte Gebiete},
	volume    = {69},
	year      = {1985},
	pages     = {279--322},
	doi       = {10.1007/BF02450284}
}

@article{CarrilloChoiHauraySalem2019,
	author  = {Carrillo, Jos{\'e} A. and Choi, Young-Pil and Hauray, Maxime and Salem, Samir},
	title   = {Mean-field limit for collective behavior models with sharp sensitivity regions},
	journal = {Journal of the European Mathematical Society},
	volume  = {21},
	number  = {1},
	pages   = {121--161},
	year    = {2019},
	doi     = {10.4171/JEMS/832}
}

@article{MuchaPeszek2018,
	author  = {Mucha, Piotr B. and Peszek, Jan},
	title   = {The Cucker--Smale Equation: Singular Communication Weight, Measure-Valued Solutions and Weak-Atomic Uniqueness},
	journal = {Archive for Rational Mechanics and Analysis},
	volume  = {227},
	number  = {1},
	pages   = {273--308},
	year    = {2018},
	doi     = {10.1007/s00205-017-1160-x}
}

@article{Dob79,
  author  = {Dobrushin, R. L.},
  title   = {Vlasov equations},
  journal = {Functional Analysis and Its Applications},
  volume  = {13},
  number  = {2},
  pages   = {115--123},
  year    = {1979}
}

@article{PS17,
  title        = {The {Boltzmann--Grad} limit of a hard sphere system: analysis of the correlation error},
  author       = {Pulvirenti, Mario and Simonella, Sergio},
  journal      = {Inventiones Mathematicae},
  volume       = {207},
  number       = {3},
  pages        = {1135--1237},
  year         = {2017},
}

@article{BPS13,
  title={From particle systems to the {L}andau equation: a consistency result},
  author={Boblylev, AV and Pulvirenti, Mario and Saffirio, Chiara},
  journal={Communications in Mathematical Physics},
  volume={319},
  pages={683--702},
  year={2013},
  publisher={Springer}
}

@article{HLL09,
author = {Seung-Yeal Ha and Kiseop Lee and Doron Levy},
title = {{Emergence of time-asymptotic flocking in a stochastic Cucker-Smale system}},
volume = {7},
journal = {Communications in Mathematical Sciences},
number = {2},
publisher = {International Press of Boston},
pages = {453 -- 469},
year = {2009},
}

\end{document}